\documentclass[10pt,eqno]{article}
\usepackage{latexsym}
\usepackage{amsfonts,amsmath,amssymb}                                                                                                                                                                                                                                                                                                                                                                                                                                                                                                                                                                                      
\usepackage{amsthm}
\newtheorem{theorem}{Theorem}[section]
\newtheorem{lemma}{Lemma}[section]

\usepackage[pdftex]{graphicx}
\usepackage[latin1]{inputenc}
\usepackage{epstopdf}
\usepackage{setspace}
\usepackage{tikz}
\usepackage[toc]{appendix}

\numberwithin{equation}{section}
\numberwithin{theorem}{section}

\numberwithin{proposition}{section}
\numberwithin{lemma}{section}
\numberwithin{remark}{section}
\numberwithin{equation}{section}
\begin{document}
\doublespacing
\setcounter{secnumdepth}{3}
%\numberwithin{equation}{section}
\newcommand{\bbox}{\vrule height.6em width.6em
depth0em} %%%%% Black Box
\parindent0in
\def\lhead{ vv }
\thispagestyle{empty}
\setcounter{page}{1}
\noindent
\def\lhead{vv }
\def\rhead{ Holder}
\thispagestyle{empty}
\setcounter{page}{1}
\noindent
\begin{center}
{\bf \large { H\"older estimates  of weak solutions to  degenerate chemotaxis systems with a  source term.}} \\
\vspace*{1cm} 
{\bf 
{ M.Marras \footnote{Dipartimento di Matematica e Informatica, Università di Cagliari, via Ospedale 72, 09124 Cagliari (Italy), mmarras@unica.it},
F.Ragnedda \footnote{ Facoltà di Ingegneria e Architettura,  Università di Cagliari, Viale Merello 92, 09123 Cagliari,  ragnedda@unica.it}, 
S.Vernier-Piro \footnote{Facoltà di Ingegneria e Architettura,  Università di Cagliari, Viale Merello 92, 09123 Cagliari (Italy), svernier@unica.it} 
 and V.Vespri \footnote
{Dipartimento di Matematica ed Informatica "U. Dini'', Università di Firenze, viale Morgagni 67/a, 50134 Firenze (Italy)
vincenzo.vespri@unifi.it}  }}\\
\end{center}
$$ \quad             In \ memory \ of \  our \ friend \  Emmanuele \  DiBenedetto$$

\vspace*{1cm}
\thispagestyle{empty}
\setcounter{page}{1}
\noindent

\begin{abstract}
In this note we consider degenerate chemotaxis  systems with porous media type  diffusion and  
a source term  satisfying  the Hadamard growth condition. We prove the  H\"older regularity for bounded solutions to parabolic-parabolic as well as for elliptic-parabolic chemotaxis systems.
\end{abstract}

\section{Introduction}
Let us consider the following class of degenerate chemotaxis  systems  \begin{equation}
\label{first equ}
\left\{ \begin{array}{l}
u_{t}=
\textrm{div}(\nabla u^m) 
   - \chi \textrm{div}(u^{{\mathfrak q}-1}  \nabla v) + B(x,t,u, \nabla u),   \  
   {\rm in} \  \mathbb{R}^N\! \times \! (t>0),    \\[6pt]
{\tilde  \tau} v_{t}=  \Delta v- a v+  \ u ,  \quad \quad \quad \quad \quad 
\quad \quad \quad \quad \quad \quad \quad \quad  {\rm in} \  \mathbb{R}^N\! \times \! (t>0),  \\[6pt]
u(x,0)=u_0(x)\geq 0, \ \  v(x,0)=v_0(x)\geq 0, \quad\quad \quad \quad    {\rm in} \ \ \mathbb{R}^N,
\end{array} \right.
\end{equation}
with  $N\geq2, m \geq 1, {\mathfrak q}\geq  \max\{\frac {m+1}2,  2\}$ and  $\ a, \chi>0$.  The constant $\tilde \tau$ is taken nonnegative.  When the constant  $ \tilde \tau =0 $ we are in the parabolic-elliptic case and,  when   $ \tilde \tau>0$, we are in the parabolic-parabolic case. In the latter case, WLOG, we may assume  $ \tilde \tau =1 $ . The  initial data $(u_0(x), v_0(x))$ satisfy
\begin{equation}
\label{init data}
\left\{ \begin{array}{l}
u_0(x) \geq 0 , \  u_0(x) \in L^{\infty}(\mathbb{R}^N) \cap L^1(\mathbb{R}^N),  \  u_0^m \in H^1,  \\[6pt]
 v_0(x) \geq 0, \  v_0(x) \in L^{1}(\mathbb{R}^N) \cap W^{1,p}(\mathbb{R}^N),
\end{array} \right.
\end{equation}
with $1<p<\infty$.

 A pair $(u,v)$ of non negative measurable functions defined in $\mathbb{R}^N\times [0,T], \ T>0$  is a local weak solution to  \eqref{first equ} if
\begin{equation*}
\begin{aligned} 
\! u\in L^{\infty}(0,T;L^p(\mathbb{R}^N)), \ u^m\in L^{2 }(0,T;H^1(\mathbb{R}^N)), \ v\in L^{\infty}(0,T; H^1(\mathbb{R}^N)), 
\end{aligned}
\end{equation*}
 and $(u,v)$ satisfies \eqref{first equ} in the sense that for every compact set ${\cal{K}} \subset \mathbb{R}^N$ and  every time interval
 $[t_1, t_2] \subset [0,T]$ one has
 \begin{equation}
 \label{weaku}
 \begin{aligned} 
&\int_{{\cal{K}}  } u \psi  dx \Bigg|_{t_1}^{t_2} \! + \! \int_{t_1}^{t_2}  \!\int_{{\cal{K}}} \! \Big[ -u \psi _t + ( \nabla u^m, \nabla \psi ) - \chi u^{\mathfrak q-1} 
( \nabla v, \nabla \psi )
    \Big] dx \ dt \\
& =  \int_{t_1}^{t_2}   \int_{{\cal{K}}} B(x,t,u, \nabla u)
 \psi \ dx dt;
 \end{aligned}
\end{equation}

 \begin{equation}
 \label{weakv}
 \begin{aligned} 
 & \! \! \int_{{\cal{K}}} \! \tilde \tau v \psi   dx \Bigg|_{t_1}^{t_2} \! + \!  \int_{t_1}^{t_2} \! \! \int_{{\cal{K}}} \! \Big[ \! - \tilde \tau v \psi _t + ( \nabla v, \nabla \psi )
   \Big] dx dt \! = \!\! \int_{t_1}^{t_2} \! \! \int_{{\cal{K}}} \! (-a v+u) \psi  \ dx dt, \\
\end{aligned}
\end{equation}
for all locally bounded non negative testing function 
$\psi \in W^{1,2}_{loc}(0,T;L^2({\cal{K}}))\cap  L^{p}_{loc}(0,T; W^{1,p}_{0}
({\cal{K}})).$\\

In the last years,  there was a growing interest in the chemotaxis systems. We recall that  Keller and Segel in the seminal paper  \cite{KS} proposed a mathematical model describing the aggregation process of
amoebae by chemotaxis. For such a reason, nowadays, such kind of systems are named Keller-Segel  in their honour.
Recently many authors studied systems with porous medium-type diffusion and with a power factor in the drift term (see   \cite {IY}, \cite {IY2}, \cite{IY3},  \cite{KL}, \cite{SK} and the references therein).
In this note,  we consider  a degenerate  chemotaxis model   with porous media type  diffusion  with $m>1$. When  $m=1,\ {\mathfrak q}=2, \ B=0$,  the system \eqref{first equ} is reduced to the classical Keller-Segel system.
 In our model,  the diffusion of the cells  ($\textrm{div}(\nabla u^m) $) depends only on own density and degenerates when $ u=0$,
$m$ denotes the intensity of diffusion and the exponent  $ {\mathfrak q}$ in the drift term takes in account the nonlinear aspects of the biological phenomenon. Moreover  we assume $\chi>0$ which  means that   the cells move toward the increasing signal concentration (chemoattractant).  For sake of  simplicity, we take  $\chi=1$.\\
 This model relies on the presence of the 
 the source term  $B$  which describes the growth of the cells. Some   experimental evidences (see \cite{BBTW}) show that  $B$ is a nonlinear term and 
   satisfies {\it natural} or Hadamard growth condition, more precisely  it  satisfies the inequality  (see  \cite{DIB},  \cite{DUV})
\begin{equation}\label{source} 
 |B(x,t,u, \nabla u)|     \leq C|\nabla u^m|^2 + \phi(x,t), \quad C>0,
\end{equation}
 with $\phi(x,t)$  in the parabolic space $ L^{q,r}_{\mathbb{R}^N\times (t>0)}=L^r(0,\infty; L^q(\mathbb{R}^N)).$ The presence of this term makes more challenging  the math approach of  this system  (see  the monograph of Giaquinta \cite{G}).\\

In literature a great part of the results concerns the case $B=0$ where, depending upon the choice of $m$  and $q$, it is  possible to  find initial data  for which we have global existence and initial data for which  blow up in finite time occurs
 (see \cite {IY},  \cite {IY2},  \cite {IY3}, 
 \cite {W}  and references therein).
To our knowledge, in   the more general case $B\neq 0$,   the global existence of the solutions and the blow up phenomenon are not still studied in fully detail. Hence  this argument will be occasion  of our next future studies. For the above  reason, we will work  in a bounded time interval $[0,T]$ i.e.  before the eventual  blow-up time, assuming, therefore, that the solution $u$ remains bounded.\\ 

 For the  solutions to \eqref{first equ} 
with $B=0$ and $\tilde \tau=1,$ Ishida  and  Yokota  in  \cite{IY}
proved that a weak solution exists globally  when ${\mathfrak q} < m + {\frac 2N}$ without restriction on the size of initial data, improving both Sugiyama (\cite{SK2})  and Sugiyama and Kunii results (\cite{SK})   where  $q \leq m$ was assumed.
In \cite{IY2},  the Authors  established the global existence of  weak solutions with small initial data   when ${\mathfrak q} \geq  m + {\frac 2N}$, while 
in \cite{W2} Winkler proved that
there are initial data such that if  ${\mathfrak q} \geq  m + {\frac 2N}$ the solution blows up in finite time. Moreover, in \cite{IY3} uniform boundedness of nonnegative solutions  was derived assuming  ${\mathfrak q} < m + {\frac 2N}$. \\
If $B=0$ and $ \tilde \tau=0,$   the existence in large of the solutions  was proved in the case of    $ m> {\mathfrak q}- \frac  2N $  without  any restrictions on initial data and in the case  $1\leq m \leq {\mathfrak q}- \frac 2N $   only for small initial data (\cite{SK}).  For (local and global)  existence and nonexistence of solutions to  different classes of Keller-Segel type system we refer to \cite{BBTW}.\\
If $B\neq 0$ and $ \tilde \tau=0, m=1,q=2$, in  \cite{MVp} the authors investigated blow-up phenomena  and  obtained a safe time interval  of existence for the solutions  by deriving a lower bound of the blow-up time. \\
In  \cite{KL},  following the De Giorgi approach,
Kim and Lee 
proved regularity and uniqueness  results for solutions  to  degenerate  chemotaxis  parabolic-parabolic system  assuming that the source term is vanishing.

In this paper  we focus our attention only on 
 the local H\"older regularity  of  the solution $(u,v).$ 
 More precisely, we give a unitary  and more organic proof 
 that allows us   to treat in the same framework a more general equation 
(with source term)  either in  the parabolic-parabolic case  and  in the parabolic-elliptic case.\\
 
 Our approach is based on 
 suitable a-priori estimates on the function $v$ that solves the second equation of \eqref{first equ}  and  on   De Giorgi-DiBenedetto approach ( \cite{degiorgi}, \cite{DIB},  \cite{DBGVb}) for proving regularity of $u^m$. The regularity of $v$,  either for $\tilde \tau=1$ or  $\tilde \tau=0$,
follows in a straightforward way  from classical results theory (see, for instance, \cite{DIB}, \cite{GT}, \cite{P}). The proof has however some remarkable differences from the classical approach by DiBenedetto. In this paper we focus our attention only on the real novelties. When the modifications are based only on technicalities, we will postpone  the proof in the Appendix and when it is a structural modification we quote the corresponding papers  by DiBenedetto.

 Our main result  is
 \begin{theorem}(Regularity)\\
\label{IntReg}
 Let $u$ be a locally bounded local weak solution of \eqref{first equ} and let
$B$ satisfy  \eqref{source},
then $(x,t) \rightarrow u(x, t)$ is H\"older continuous in $R^N\times (0,T)$ 
and  there exists $\alpha_o \in (0,1)$ such that,
  for every $\varepsilon >0$,  there exists a constant $\gamma(\varepsilon) >0$ such that
 
 \begin{equation*}\label{interior}
 | u^m(x_1, t_1) - u^m(x_2, t_2)| \leq \gamma(\varepsilon) (|x_1 - x_2|^{\alpha_o} + |t_1 - t_2|^{\frac{\alpha_o} 2}) ,
 \end{equation*}
for every pair of points $(x_1, t_1), (x_2, t_2) \in  R^N\times (\varepsilon,T)$.
 \end{theorem}

The scheme of this paper is the following:
in Section \ref{pre} we present some preliminary lemmas  that will be used to prove our main results. 

In Sections \ref{first pp} and  \ref{second pp} we consider the parabolic-parabolic chemotaxis system.  We study the behavior of $u^m$ near the ess inf $u^m$($First \  Alternative$), 
and  near the ess sup $u^m$($Second  \ Alternative$) respectively and  combining these estimates  we have  the H\"older continuity of the solution.
The  Section \ref{Holder} is devoted  to  the study of the H\"older regularity for the parabolic-elliptic chemotaxis system.
 At the end,
 we  insert an Appendix, where  the proofs of a few of lemmas and technical arguments are collected.
 
  %%%%%%%%%%%%%%
\section{Preliminary results}
\label{pre}

In this section we present some results we will use in the sequel.								

In $\mathbb{R}^N$,  define the N-dimensional cube centered at the origin and wedge  $2R$:
$ K_R=\{ x\in \mathbb{R}^N / \underset{1\leq i \leq N}{\max} |x_i|<R\}
$ and let $ |K_R|$  be its measure. Let $Q_R(T)=:K_R\times [0,T]$.
Consider the parabolic space $L^{q,r}(Q_R(T))$
with the norm
$$
||w||_{L^{q,r}(Q_R(T))} \equiv \Big(\int_0^T\Big( \int_{K_R } |w|^qdx \Big)^{\frac rq }    dt \Big)^{\frac 1r } <\infty.
$$
Moreover, for  $p>1$, $w$ belongs to the space 
$$
V^p(Q_R(T)) \equiv L^{\infty }(0,T;L^p(K_R))\cap L^p(0,T;W^{1,p}(K_R))
$$ 
 if \ 
$
||w||_{ V^p(Q_R(T) )} \equiv \underset{(0,T)}{\text{ess\,sup}}   ||w||_{L^p(K_R)} + \ || \nabla w||_{L^p(Q_R(T)) }<\infty.
$\\
Define 
 $V_0^p(Q_R(T)) \equiv L^{\infty }(0,T;L^p(K_R))\cap L^p(0,T;W_0^{1,p}(K_R))$.
% % %

The proof of the following lemmata can be found in the monograph \cite{DIB}.

\

  \begin{lemma}(Sobolev Lemma)
\label{sobolev}
Let $ {\bar \zeta(x,t)} $ be a cut off function compactly supported in a cube $K_R, \ R >0$ and let $u(x,t)$ be defined  in $\mathbb{R}^N \times (t_1,t_2)$,  for any $t_2>t_1>0.$ \ Then  
\begin{equation*}\label{sob1}
||\bar \zeta u ||_{L^{ \frac{2N}{N-2} } (\mathbb{R}^N)   }
    \leq  C ||\nabla( \bar \zeta u)||_{L^2(\mathbb{R}^N)};
\end{equation*}
and, for some $C>0,$
\begin{equation}
\label{sob2}
\begin{aligned} 
&||\bar \zeta u ||^2_{L^2(t_1,t_2;L^2(\mathbb{R}^N))}
 \leq \\
 &
  C \Big(  \underset{(t_1\leq t \leq t_2)}{\text{sup}}  || \bar \zeta u ||^2_{L^2(\mathbb{R}^N) }+ \ ||\nabla( \bar \zeta u)||^2_{L^2(t_1,t_2;L^2(\mathbb{R}^N))}\Big)
   |\{ \bar \zeta u>0 \}|^{\frac{2}{N+2}}.
\end{aligned} 
\end{equation}
\end{lemma}
For the  parabolic spaces, the following embedding inequality holds.
\begin{lemma}(Embedding  Lemma)
\label {embeded}
There exists a positive constant $\gamma_1=\gamma_1(N,p)$ such that for every function $w \in V_0^p(Q_R(T) )$
\begin{equation}\label{embed}
||w||_{L^{q,r}(Q_R(T))}  \leq \gamma_1 ||w||_{V_0^p(Q_R(T))},
\end{equation}
where $p,q,r$ satisfy the relation
\begin{equation}
\label{prq}
\frac 1{r} + \frac N{p q} = \frac N{ p^2},
\end{equation}
and in the case  $1\leq p <N$,  the admissible range is
$q\in  \Big[p, \frac { Np} { N-p} \Big], \  r\in[p,\infty].$

\end{lemma}

\

%%%%%%%%%%%%     	 
\begin{lemma}(Fast geometric convergence Lemma).\\
\label{fastconv}
Let  $(X_i) \ and \ (Y_i), \ i=0,1,...$ be two sequences of positive numbers satisfying the recursive inequalities
\begin{equation*}
\label{fast}
X_{i+1} \leq c \ {\rm b}^i (X_i^{ 1+\hat\alpha}+ X_i^{\hat \alpha}   \ Y_i^{ 1+\kappa});
 \quad \quad
Y_{i+1} \leq  c\  {\rm b}^i (X_i+   \ Y_i^{ 1+\kappa}) \\ 
\end{equation*} 
with $c,{\rm b}>1$  and $ \kappa, \hat \alpha>0$  given numbers. If
\begin{equation*}
X_0+   \ Y_0^{ 1+\kappa} \leq (2c)^{- \frac{1+\kappa} {\sigma } }  {\rm b}^{- \frac{1+\kappa} {\sigma^2 } }, \quad \sigma = \min(\kappa, \hat \alpha),
\end{equation*} 
then $(X_i) \ and \ (Y_i) \rightarrow 0$ as $i \to \infty.  $
\end{lemma}

{\bf Steklov averages}\\
Since the solutions of  \eqref{first equ}  possess a modest degree of regularity  in the time variable, we utilize the Steklov average $u_h$ of the weak solution $u$, for $h>0$:
 \begin{equation*}
u_h(\cdot, t)= \frac 1h \int_t^{t+h} u(\cdot,\tau) d\tau.
\end{equation*}
For a complete statement of Steklov averages and for their convergence to  $ u$, as $h\to 0$, we refer the reader to  \cite{DIB} and \cite{DBGVb}.
Then we point out an alternative formulation of weak solution to \eqref{first equ}:
fix $t\in(0,T)$, let $h>0$  with $ 0<t<t+h<T$ and replace in \eqref{weaku}
$t_1$ with $t$ and $t_2$ with $t+h$; choosing  a test function $\psi$ independent on  $\tau \in (t,t+h)$, dividing by $h$, using the Steklov averages   we get
\begin{equation}
 \label{weak steklov}
 \begin{aligned} 
& \int_{{\cal{K}}\times t }[ (u_h)_t \psi    + ( (\nabla u^m)_h, \nabla \psi ) - \chi (u^{\mathfrak q-1} 
 \nabla v)_h  \nabla \psi ]
     dx  d \tau \\
     &=     \int_{{\cal{K}} \times t} B(x,t,u, \nabla u) 
 \psi \ dx d \tau,\\
\end{aligned}
\end{equation}
for all locally bounded non negative testing function 
$\psi \in W^{1,2}_{loc}(0,T;L^2({\cal{K}}))\cap  L^{p}_{loc}(0,T; W^{1,p}_{0}
({\cal{K}})).$\\
Integrating over $[t_1,t_2] $ and letting $h\to 0$  \eqref{weak steklov} gives \eqref{weaku}.\\
%%%%%%%%%% 	

An essential tool for the regularity are the energy estimates. Define $ (k- u^m)_+$ ($ (k- u^m)_-$, resp.) as  $ k- u^m$ if $k>u^m$ ($ u^m- k$ if $k<u^m$, resp.) and 0 otherwise. Here we state these estimates only for $(k-u^m)_+,$ (omitting the sign +), being the other case specular. Let  \,$q,\, r > 1,\,\,0<\kappa < \frac{2}{N}$\,\,and introduce \,\, $\tilde q, \tilde r$\,\, related to \,\,$q, r, \,\,\kappa$\,\, by the formulas: 
\begin{equation}\label{eq.rtil} 
 1-\frac 1{q}   = \frac{2(1+\kappa)} { \tilde q} , \ \ 1-\frac 1{r}   = \frac{2(1+\kappa) } { \tilde r}.
 \end{equation}

 % EE
\begin{lemma}(Local Energy estimates) 
\label{EnEst}
Let $t_1<t_2, \ {\mathfrak q}>2, \ m>1$ and let  $(u,v)$ be a locally bounded  weak solution  of problem \eqref{first equ}.\\
There exist  constants  $\gamma^*>0$ and ${\bf C}>0$  depending only upon the data and $||u||_{L^{\infty}(R^N\times (\varepsilon,T))},$  such that for a cut-off function $\eta$ compactly supported in $K_R$ and for every  level $k$,
\begin{equation}\label{EE}
\begin{aligned}
&\int_{K_R\times t_2  } \eta^2 \Big[ \int_0^{ (k-u^m)} \Phi(\xi) \ d\xi  \Big] dx + \gamma^*\int_{t_1}^{t_2} \int_{K_R}
|\nabla (\eta (k-u^m)|^2dx dt \\
& \leq {\bf C} m  \! \Big( \int_{t_1}^{t_2}\!\! \int_{K_R} (k-u^m)^2 |\nabla \eta|^2 dx dt 
+ \! \int_{t_1}^{t_2}\!\! \int_{K_R} \!\! \Big[  \int_0^{ (k-u^m)}\!\!\! \Phi(\xi ) d\xi \Big] |\eta \eta_t| dx dt\\
&+ \int_{K_R\times t_1 } \eta^2 \Big[  \int_0^{ (k-u^m)} \Phi(\xi) d\xi \Big] dx +    \Big[ \int_{t_1}^{t_2}( |A_{k,R}|^{\frac {\tilde r } {\tilde q }}dt \Big]^{\frac{2 (1+k) } {\tilde r}}\Big)
\end{aligned}
\end{equation} 
with  $\Phi(\xi )=: (k- \xi)^{\frac 1m -1 } \xi  $   and  $ A_{k,R}(t)= \{x \in K_R : (k-u^m)>0 \}.$
\end{lemma}
The proof of this estimate follows an argument similar to the one by 
DiBenedetto and  for the reader's convenience we  give a proof in the Appendix \ref{cap:Proof of Lemma}.

%%%%%%%%%%
Let us continue  this section with a priori estimates on the 
$L^p-L^{p'}$ norm  of the solutions of evolution equations with
 \begin{equation}
\label{pp'}   
 1\leq p' \leq p\leq \infty, \qquad \qquad  \frac 1{p'}- \frac 1{p} <  \frac 1{N}.
\end{equation}

 Consider   the following Cauchy problem :
 \begin{equation}
 \label{heat}
 \left\{ \begin{array}{l}
\begin{aligned}
& v_t= \Delta v - av +w , \quad (x,t)\in \mathbb{R}^N\times (t>0), \\
&v(x,0)=v_0(x), \quad  \quad x\in \mathbb{R}^N,
\end{aligned} 
\end{array} \right. 
 \end{equation}
then by classical $L^p$ maximal regularity properties (see for instance \cite{HP}, see also  \cite{SK}) we have:
%%%%%%%%%%
\begin{lemma}(Heat)
\label{Heat}
Let $v$ be the solution to \eqref{heat}. If $v_0 \in W^{1,p}(\mathbb{R}^N) $ and $ w\in L^{\infty}(0,\infty; L^{p'}(\mathbb{R}^N)),$ with $p,p'$ in \eqref{pp'}, then for $t\in [0,\infty)$, there exist  positive constants $C_0, \bar C_0$, depending on $p, p'$ and $N$  such that
  \begin{equation}
 \label{grad v parab}
 \left\{ \begin{array}{l}
\begin{aligned} 
||v(t)||_{L^p(\mathbb{R}^N)}&\leq   \ ||v_0||_{L^p(\mathbb{R}^N)} + C_0  ||w(\tau)||_{L^\infty(0,T; L^{p'}(\mathbb{R}^N))},\\
||\nabla v(t)||_{L^p(\mathbb{R}^N) } & \leq  ||\nabla v_0||_{L^{p}(\mathbb{R}^N) }+  \bar C_0  ||w(\tau)||_{L^\infty(0,T; L^{p'}(\mathbb{R}^N))}.
\end{aligned}
\end{array} \right.
\end{equation}
\end{lemma}

We conclude this section recalling the preliminaries necessary for the regularity machinery. \\
Consider  $0<R<1$ sufficiently small. Set  $Q(2R, R^{2-\varepsilon})=: K_{2R} \times [-R^{2-\varepsilon}, 0],$ where $\epsilon$ is a positive number to be chosen later. We set
\begin{equation}\label{eq.1}
 \begin{aligned} 
&\mu_+ \! = \! \! \underset{Q (2R, R^{2-\varepsilon})} {\text{ ess\,sup}}u^m,\,\,\mu_- \! = \! \!  \underset{Q (2R, R^{2-\varepsilon})}{\text{ ess\,inf}}u^m, \quad
\  \omega= \! \!  \underset{Q (2R, R^{2-\varepsilon})}{\text{ ess\,osc}}\! u^m \! =\! \mu_+ \! - \mu_-.\\
\end{aligned}
\end{equation}

 Let
$s_0$  be  the smallest integer such that
\begin{equation}
\label{s0}
\theta_0= \frac{\omega}{2^{s_0}}<1, \  \ \alpha = 1-\frac{1}{m}, \ \ a_0 =  \frac{\omega}{A}
\end{equation}
with $ A>2^{s_0} , A$ a positive constant to be determined later. \ 
 We introduce sub cylinders  with center at $(0,\bar t): Q^{\bar t}(R, \theta_0^{-\alpha} R^2) = K_R\times [\bar t- \hat \theta, \bar t  ]  =:Q^{\bar t}_R(\hat \theta)$, with $\hat \theta= \theta_0^{-\alpha} R^2$     and 
\begin{equation}
\label{Qt} 
 Q(R, a_0^{-\alpha}R^2)=:Q_ R(\hat a), \quad { \rm with} \quad  Q^{\bar t}_R(\hat \theta)  \! \subset \!
Q_R(\hat a). 
\end{equation}

WLOG, assume that
$
(\frac{\omega}{A})^\alpha> R^\varepsilon.
$ 
 If for any\, $ R <1  $\, this does not hold,  we have \,\,$\omega \leq A R^{\frac{\varepsilon}{\alpha}} $ and there is nothing to prove since the oscillation is then comparable  to the radius. \ Next, if $\omega > A R^{\frac{\varepsilon}{\alpha}},$
  we prove that the oscillation of $u^m$ is reduced by a fixed factor in the set $Q_R(\hat \theta),$ by analyzing two complementary alternatives. 
In the first alternative, let us   assume that there exists  a subcylinder $ Q^{\bar t}_R(\hat \theta) $ where $u^m$ is away from its infimum. 
Under such hypothesis,  we are able to prove  that the oscillation decreases  
of a fixed factor  in that sub cylinder. Then, by using the so-called expansion of positivity in time, we are able to transport that information to the top of the original cylinder, in the "right "sub cylinder.
In the second alternative  we 
examine the case when   in all the cylinders $ Q^{\bar t}_R(\hat \theta) $,   $u^m$ is essentially away from its supremum and we prove that the oscillation decreases by a fixed factor also in this case  directly in the whole cylinder.\\ 
%%%%%%%%%%%%%%%%%%%%%%%%%%%%%%%% 
 
%%%%%%%%%%
\section{ $1^{st} $ alternative in the parabolic-parabolic case}
\label{first pp}

In this section we examine the first alternative  with $\tilde \tau=1$ in \eqref{first equ}: i.e.

there exists a cylinder $ Q^{\bar t}_R(\hat \theta) $ where
\begin{equation}\label{first}
\frac { \big| \{(x,t)\in Q^{\bar t}_R(\hat \theta): 
u^m < \mu_- + \frac \omega{2^{s_0+1} }\} \big|} {|Q^{\bar t}_R(\hat \theta)| } \leq \nu
\end{equation}  
with the positive constant $\nu$ to be defined later.\\

To work in this context, we 
need the so-called {\it critical mass} and  {\it expansion in time}  lemmata. For more details  about these definitions, see, for instance, the review paper \cite{DMoV}.

%%%%%%%%%%%
\begin{lemma}(Critical mass lemma)
\label{teo4}
Let us consider the cylinder $Q^{\bar t}_R(\hat \theta)$  defined in \eqref{Qt} and let $s_0$ be defined in  \eqref{s0}. Then  there exists a number $\nu\in (0,1)$ such that, if
\begin{equation}\label{eq.5}
 | \{ (x,t) \in Q^{\bar t}_R(\hat \theta) : \ u^m < \mu_- + \frac \omega{2^{s_0} }   \}| \leq \nu \ | Q^{\bar t}_R(\hat \theta)  |,
 \end{equation} 
 then
 \begin{equation}\label{eq.1}
u^m > \mu_- + \frac \omega{2^{s_0+1} }, \quad  a.e.  \  \ in \  Q^{\bar t}_{\frac R2}(\hat \theta).
\end{equation} 
\end{lemma}
\begin{lemma}(Expansion in time lemma)
 \label{for every nu}
For every $\nu_1 \in (0,1)$, there exists  a positive integer $s_1$, depending only upon  the data and independent of $\omega, R$ such that for all $t\in  (\bar t- \hat \theta, 0)$
\begin{equation}
\label{for every}
\Big| \left\lbrace  x\in\,K_{ \frac R4}: u^m(x,t)< \mu_- + \frac{\omega}{2^{s_1}}\right\rbrace \Big | \leq  \, \nu_1  \big|K_{ \frac R4}\big|,
\end{equation} 
 \end{lemma} 
The proof of the two lemmata  cannot be derived in a straightforward way from what is known in literature and for this reason we prefer to give a detailed proof of them. See Appendix  \ref{cap:Estimate Phi} for the Critical Mass Lemma
 and Appendix  \ref{cap:forevery} for the Expansion in time Lemma.

%%%%%%%%%%%%%
 Now, using Lemma \ref{for every nu} and 
  following the same argument
developed in  Chapter III, Section 6 of \cite{DIB},   in the   subcylinder $K_{\frac R4}\times 
(\bar t - \theta_0^{\alpha} (\frac R4)^2, \ 0)$, one can prove 
 that the numbers $\nu_1$ and $ s_1$  of Lemma \ref{for every nu}
 can be chosen a priori depending only upon the data and independent of $\omega$ and $R$, such that we have 
\begin{equation}
\label{R8}
u^m> \mu_- + \frac \omega{2^{s_1+1 } } \quad a.e. \ (x,t) \in K_{\frac R8}\times 
\Big(\bar t -\theta_0^{\alpha} \Big(\frac R8\Big)^2, \ 0 \Big).
\end{equation}
This concludes  the $1^{st} $ alternative, because we have proved 
 the reduction of the oscillation  in that sub cylinder. Thanks to the expansion of positivity in time (for more detail about this now a classic tool, see \cite{DIB}) we are able to transport this information to the top of the original cylinder.
  In fact, we have shown that in the sub cylinder located at the top, 
 there exist  numbers
$ \eta_1=\big(1-  \frac{1 } { 2^{s_1+1}} \big), $ and $\mathcal {A}_1 >A$
  that can be determined a priori in terms of the data, such that 
 \begin{equation}
\label{oscillation1} 
either \ \quad \quad   \underset{Q_{\frac R8}(\hat \theta)}{ess \ osc } \  u^m \leq \eta_1  \ \omega \quad \quad or \quad \omega \leq \mathcal {A}_1 R^{\frac {\epsilon } \alpha}.
\end{equation}

%%%%%%%%%%%%%%%%%                   222222 Alt

\section{$2^{nd} $ alternative for the parabolic-parabolic case}
\label{second pp} 
Suppose   that the assumption  of the first alternative is violated, i.e.
for all   subcylinders  $Q^{\bar t}_{R }(\hat \theta) \subset 
Q_{R}(\hat a)$
\begin{equation}
\label{violated}
\Big| \left\lbrace (x,t)\in\,Q^{\bar t}_{R }(\hat \theta): u^m(x,t)< \mu_- + \frac{\omega}{2^{s_o}}\right\rbrace \Big|> \, \nu \big|Q_{R }(\hat \theta)\big|,
\end{equation}
We can rewrite \eqref{violated} as
\begin{equation}
\label{violated2}
\Big|\left\lbrace (x,t)  \! \in \! Q^{\bar t}_{R}(\hat \theta): \! u^m(x,t)> \mu_+ - \frac{\omega}{2^{s_o}}\right\rbrace \Big|\leq \! (1 \! - \! \nu) \big|Q_{R}(\hat \theta)\big|.
\end{equation}
%%%%%%%%%%%%%%%%%%%%%%%% 
In order to estimate the measure of the set where $u^m(x, t)> \mu_+ - \frac{\omega}{2^{s_o}} $ within $K_R,$
we use the following Lemma proven in \cite{DIB}.
 %%%%%%%%%%%%%%%%%%%%%%        
\begin{lemma}\label{Lem t star}(t* Lemma). 
Fix  $Q^{\bar t}_{R}(\hat \theta) \subset 
Q_{R}(\hat a)$  and assume that \eqref{violated2} holds. There exists a time level \,$\,t^*
\in [\bar t -\theta_0^{-\alpha}R^2, \bar t -\frac{\nu}{2}\theta_0^{-\alpha}R^2] $ such that
\begin{equation}
\label{t^*}
\Big| \! \left\lbrace x\in K_R \! : u^m(x,t^*) \! > \mu_+ \! - \frac{\omega}{2^{s_o}}\right\rbrace \! \Big| \! \leq  \! \frac{1-\nu}{1-\frac{\nu}{2}} \  |K_R|.
\end{equation}
\end{lemma}

 Now let us   evaluate the measure of the set
$\left\lbrace x\in\,K_R: u^m(x,t)> \mu_+ - \frac{\omega}{2^{s_1}}\right\rbrace,$ $t\,\in[t^*, 0],$
where $ s_1\! > s_o $ is an integer to be fixed later on. \\
At this end, pick 
$ H= \underset{K_R\times [t^*,0]}{\text{ ess sup}}\left( u^m-(\mu_+ - \frac{\omega}{2^{s_o}})\right).$
In $\, K_R \times [t^*,0] ,$  consider the function
\begin{equation}
 \varphi(H)\,=\varphi(u^m) = log^+ \left( \frac{H}{H- (u^m -k) +c} \right),\ k = \mu_+ -\frac{\omega}{2^{s_o}}, \ c = \frac{\omega}{2^{s_1}}.
\end{equation}
Note that
\begin{equation}   
\begin{aligned}
&\varphi'(u^m)  = \frac{1}{H-(u^m -k)
 + c} > \frac{H}{H- (u^m -k) +c}\geq 1,  \\
& \varphi''(u^m) =  (\varphi'(u^m))^2, \  \ 
\nabla \varphi= \varphi'\nabla u^m,\,\,\,\nabla\varphi' =\varphi'^{2}\nabla u^m.
 \end{aligned}
  \end{equation}
  
  %%%%%%%%%%%%%%%%%%%%    
   \begin{lemma}(Logarithmic Lemma)\\
  \label{log}  
  There exists  an integer $s_1\,> s_0$,\, such that if 
  $ H > \frac{\omega}{2^{s_1}},$      
  then
  $$\Big|\left\lbrace x\in\,K_R: u^m(x,t)> \mu_+ - \frac{\omega}{2^{s_1}}\right\rbrace\Big|  \,\leq \Big(1- \ \frac{\nu^2}{4}\Big) \ |K_R|,\,\,\forall t\,\in[t^*, 0].$$
\end{lemma}
 
 \begin{proof}
Split \,\,$K_R = K_{(1-\sigma)R} +K_{\sigma R}, 0 <\sigma< 1.$
 Making use of the definition of weak solution, via the Steklov average, 
  pick as test function $ \,\,\psi(u_h^m) = m(u_h^m)^\alpha(\varphi^2(u_h^m))' \xi^2,$ where $\xi = \xi(x)$ \  is a cut-off  function, $ \xi = 1\,\text{in the cube }\,\, K_{(1-\sigma)R},  \\  \ \xi =0 \,\,\text{on} \ \partial K_R,\,
  |\nabla \xi|\leq \frac{2}{\sigma R}.$\\\
  For any \,\, $t^*\leq t\leq 0$,
 letting  \,$h\rightarrow \,0$, we get
 \begin{equation}
\begin{aligned} 
\iint_{K_R\times (t^*\!,t)} \! \! (\varphi^2(u^m)  \xi^2)_t dx d\tau = \!\int_{K_R\times\{t\}} \!\varphi^2(u^m)  \xi^2dx - \int_{K_R\times\{t^*\}} \! \varphi^2(u^m)  \xi^2dx.
\end{aligned}
 \end{equation} 
Define
\begin{equation}
\begin{aligned} 
&-\int_{t^*}^{t}\int_{K_R}(\nabla u^m)\cdot\nabla  \psi dxd\tau 
+ \int_{t^*}^{t}\int_{K_R}(u^{q-1}\nabla v)\cdot\nabla  \psi dx d\tau \\
&+ \  \int_{t^*}^{t}\int_{K_R} C|\nabla u^m|^2  \psi \ dx d\tau + \int_{t^*}^{t}\int_{K_R} \phi \psi dxd\tau\\
&=: \ -\int_{t^*}^{t}Id\tau\,+\int_{t^*}^{t}Jdt\tau + \int_{t^*}^{t}L_1 d\tau+ \int_{t^*}^{t}L_2 d\tau
\end{aligned}
  \end{equation}
and,   from the definition of weak solution, get the following inequality
 \begin{equation}
 \label{weakloglemma}
\begin{aligned} 
&\int_{K_{(1-\sigma) R}\times\{0\}} \ \varphi^2(u^m)\xi^2  \ dx\leq
\int_{K_{ R}\times\{0\}}\ \varphi^2(u^m)  \xi^2 \ dx  \\
& = \int_{K_R\times\{t^*\}} \ \varphi^2((u^m)  \xi^2 \ dx +\int_{t^*}^{0}(-I\,+J +L_1+L_2)  \ dt.
 \end{aligned}
\end{equation}
%%%%%%%%%%%%%%%%%%  

  To estimate the term\,\, $\int_{K_{(1-\sigma ) R}\times\{t\}} \varphi^2(u^m)  \xi^2 \ dx $\,\, let\\
  $\bar P= \{x\in K_{(1-\sigma) R}:u^m(x, t)> \mu_+  -\frac{\omega}{2^{s_1}}\}$ , with
 $ t^*<t<0.$  \\
  Using the notation  $  \varphi(u^m) = \varphi(H),$ in \,\,$\bar P$ \  we have 
 $ \varphi(H) \geq   (s_1 - s_0 -1) \ln2  $ (to prove this estimate we refer the reader to Chapter II, Section 3-(ii) in \cite{DIB}). 
%%%%%%%%%%%%%%%%%
Hence by Lemma  \ref{Lem t star} 
  \begin{equation}
  \label{second term}
  \int_{K_{R}\times\{t^*\}}\varphi^2(u^m)  \xi^2 \ dx \leq \ln^2 2(s_1 -s_0)^2 \  \Big( \frac{1-\nu}{1-\frac{\nu}{2}} \Big) \ |K_R|.
  \end{equation}
and

 \begin{equation}
 \label{first term} 
\begin{aligned} 
&\int_{\bar P\times\{t\}} \varphi^2(u^m)dx \\
& >  (s_1 - s_0 -1)^2  \ln^2 2 \Big| \Big\{x\in K_{(1-\sigma) R}:u^m(x, t)> \mu_+  -\frac{\omega}{2^{s_1}}\Big\}\Big|;
\end{aligned}
  \end{equation}
  
 For the second term on the right of \eqref{weakloglemma} we have
the following estimate, which proof, for the sake of readability of the paper, 
  is postponed  in Appendix  \ref{cap:stimaIJL1L2}.
 \begin{equation} 
 \label{IJL1L2} 
\begin{aligned} 
&\int_{t^*}^{0}( -I +J +L_1+L_2) dt\\
&\leq 3m\,I_u(1+\mu_+ )
(1+(s_1 - s_0)\ln2  \Big( \frac{2^{s_1}}{\omega}\Big)^2 \mu_+^{\frac{2\mathfrak q-3}{m}} \Big(\frac{2^{s_0}}{A} \Big)^{\frac{2 \alpha(1+k) } {\tilde r } } |K_R| \\
&+ 2m \  \mu_+^{m-1} (s_1-s_0)\ln 2 \ \frac{2^{s_1} } {\omega} || \phi||_{q,r,Q_{R}(\hat \theta)} \ \Big( \frac{2^{s_0} }{ A}\Big)^{\frac{2 \alpha(1+k) } {\tilde r } }|K_R|
\\
&+6m \gamma \frac{(2^{s_0})^\alpha}{ \omega^\alpha}  \frac{(s_1 -s_0 )\ln2 \  \mu_+^{m-1}}{ \sigma^2} |K_R|. 
\end{aligned}  
\end{equation}
 By inserting \eqref {first term}, 
  \eqref {second term}
  and  \eqref{IJL1L2} \  in \  \eqref{weakloglemma} and 
      dividing by 
      $ (s_1 - s_0 -1)^2 \ \ln^2 2 $, 
  we obtain
 \begin{equation}  
\begin{aligned} 
&|\{x\in K_{ R}:u^m(x,t)> \mu_+  -\frac{\omega}{2^{s_1}}\} |\\
& \leq \left( \frac{s_1 -s_0}{s_1 -s_0 -1}\right)^2\left(  \frac{1-\nu}{1-\frac{\nu}{2}}\right) |K_R|  + 
3m\, I_u(1+\mu^+ )\cdot\\
&\frac{(1+(s_1 - s_0)\ln2 )}{(s_1 -s_0 -1)^2\ln^22}\left( \frac{2^{s_1}}{\omega}\right)^2 \ \mu_+^{\frac{2q-3}{m}} \,\Big(\frac{2^{s_0}}{A}\Big)^{\frac{2 \alpha(1+k) } {\tilde r } } |K_R| \\
&+ 2m \frac{ \mu_+^{m-1} (s_1-s_0)} { (s_1 -s_0 -1)^2 \ln2}  \frac{2^{s_1} } {\omega} || \phi||   \  \Big( \frac{2^{s_0} }{ A}\Big)^{\frac{2 \alpha(1+k) } {\tilde r } }|K_R|
\\
&+6\gamma m \frac{(2^{s_0})^\alpha}{ \omega^\alpha} \frac{(s_1 -s_0 )\,\mu_+^{m-1}}{ \sigma^2 \ (s_1 -s_0 -1)^2\ln2} |K_R| +N\sigma |K_R |  \\
& \equiv ( \mathsf{A}+ \mathsf{B}_1+\mathsf{B}_2+\mathsf{C}+N \sigma) \ |K_R|,
 \end{aligned}  
  \end{equation}
  where we used the fact that
 \begin{equation*}   
 \Big|  \Big\{x\in K_{ R}\!:u^m>\! \mu_+ \! -\frac{\omega}{2^{s_1}} \Big\}  \Big|\leq \!
 \Big|  \Big\{x\in K_{(1-\sigma) R}\!:u^m> \mu_+ \! -\frac{\omega}{2^{s_1}} \Big\} \Big|+N\sigma  |K_R |.
  \end{equation*}   

  Choosing  \,\,$ \sigma$\,\, such that\,\,$ N\sigma \leq    
   \frac{1}{4}\nu^2,$ 
  $ s_1$ such that\,\, $\left( \frac{s_1 -s_0}{s_1 -s_0 -1}\right)^2\leq (1-\frac{1}{2}\nu)(1+\nu)$, \ we have 
$\mathsf{A}\leq 1- \nu^2$  and  \ 
$ \mathsf{C} \leq \frac{1}{4}\nu^2$
 and for such $\sigma$ and $ \ s_1,$ let  $A $\,\, be such that  $    
 \mathsf{B}_1 \leq \frac{1}{8}\nu^2$ \ 
and
$ \mathsf{B}_2 \leq \frac 18 \nu^2$
and this implies the statement of  Lemma \ref{log}.
\end{proof}
%%%%%%

The second alternative is concluded estimating the measure of the set where $u^m(x, t)> \mu_+ - \frac{\omega}{2^{s^*}} , s^*>s_2 $ within a sub cylinder of  $Q_R(\frac {\hat a}2)$. This can be done via the following two lemmata which proofs can be deduced from Lemma  8.1 and Lemma 9.1 in \cite{DIB}.
%%%%%%%%
 \begin{lemma}
 For every $\nu^*\in(0,1),$ there exists $s^*>s_2,$ independent on $R, \omega$ such that
 \begin{equation}
 \Big|  \Big\{x\in Q_R(\frac {\hat a}2):u^m(x,t)> \mu_+  -\frac{\omega}{2^{s^*}}  \Big\}  \Big|\leq \nu^* \ |Q_R(\frac {\hat a}2) |
\end{equation} 
with $A= 2^{s^*}, \ a_0=\frac{\omega }A$.
  \end{lemma}

 %%%%%%%%%%

   \begin{lemma}
 The number $\nu^*$ (and $s^*$) can be chosen  such that
 \begin{equation}
 \label{nu}
u^m(x,t) \leq  \mu_+  -\frac{\omega}{2^{s^*+1}}, \quad a.e. \ in \  Q \Big(\frac R 2, \frac 12  \Big(\frac{\omega}{2^{s*}}\Big)^{-\alpha}
\left(\frac R2\right)^2 \Big).
\end{equation}  
 \end{lemma}

The reduction of the oscillation concludes 
 the $2^{st} $ alternative.   \qed

Following the approach by Di Benedetto (see \cite{DIB}, chap.III),
the two alternatives  imply  the H\"older continuity of $u^m$, hence
 Theorem  \ref{IntReg}  is proved in the parabolic-parabolic case.

%%%%%%%%%%%%%%%%                  HOLDER

 \section{H\"older continuity  to the parabolic-elliptic chemotaxis system }
\label {Holder}
 The aim of this section is to extend the results obtained in the previous sections for  the system  \eqref{first equ} with $\tilde \tau=1$, to   the following parabolic-elliptic  degenerate system ($\tilde \tau=0$) in $\mathbb{R}^N \times (t>0)$
\begin{equation}
\label{second equ}
\left\{ \begin{array}{l}
u_{t}=
\textrm{div}(\nabla u^m) 
   - \chi \textrm{div}(u^{{\mathfrak q}-1}  \nabla v) + B(x,t,u, \nabla u),      \\[6pt]
0=  \Delta v-  a v+  \ u ,    
\end{array} \right.
\end{equation}
with nonnegative initial data satisfying \eqref{init data}. 
Our approach is unitary and does not see the difference between the parabolic-parabolic and parabolic-elliptic cases. For this reason 
the proof of H\"older continuity of $u^m$ follows almost entirely  the steps of  the Sections \ref{first pp} and  \ref{second pp} for the parabolic-parabolic case. In this section we will focus our attention only on the main differences.
First, let us  state a-priori elliptic $L^p$ estimates.\\
%%%%%%%%%%%%%%
Consider the elliptic equation
 \begin{equation}
 \label{ell2}
- \Delta v + a v =  w , \quad  x\in \mathbb{R}^N.
 \end{equation}
 By  classical $L^p$ regularity results (\cite{CS}, \cite{CZ})

%%%%%%%%%%%%

 \begin{equation}
 \label{G}
||v(x)||_{W^{2,p}(\mathbb{R}^N)}\leq  c ||w(x)||_{L^p(\mathbb{R}^N)}
\end{equation}
where $c$ is a constant depending upon $p, N$ and $a$.

Let us start now the study of  the $ 1^{st} $\ Alternative.

%%%%%%%%%%%%%%%
In order to extend the result  in Lemma  \ref{teo4} we must consider the terms containing $|\nabla v|,$ and there, instead of \eqref {grad v parab}, we have to  use the estimate \eqref{G}. The same must be done in the analogous of   Lemma 3.2. 
\\
Let us explain some details.
To construct the sequences $X_i$ and $Y_i$ the   
Lemma \ref{EnEst} must be  applied.
More precisely, we must change the estimate  of the term with $|\nabla v|$  ( see \eqref{gradv} in Appendix A) present in
\begin{equation}\label{gradv ell 1 alt}
 \begin{aligned}
 &  \int_{t_1}^{t_2}\int_{{K_R}\bigcap\{(k-u^m) >0\}}u^{2({\mathfrak q}-1)}\eta^2|\nabla v|^2dx d t\\
&\leq \mu_+^{\frac{2({\mathfrak q}-1)}{m}}\Big(\int_{t_1}^{t_2}  \left(\int_{K_R}|\nabla v|^{2q}dx\Big)^{\frac{r}{q}}d t \right)^{\frac{1}{r}}\!\! \Big(\! \int_{t_1}^{t_2}\Big|A_{k,R} (t) \Big|^{\frac{r(q-1)} {q(r-1)} }d t \Big) ^{\frac{r-1}{r}}.
\end{aligned}
  \end{equation} 

By using  \eqref{G}    with $w=u$ and $p=2q$, we  obtain
 \begin{equation}\label{gradv ell 1 alt bis}
 \Big(\int_{t_1}^{t_2} \!\! \left(\int_{K_R}|\nabla v|^{2q}dx\Big)^{\frac{r}{q}}d t \right)^{\frac{1}{r}} 
 \leq c  \Big(\int_{t_1}^{t_2}   ||u||_{2q}^{2r} d t  \Big)^{\frac 1 r }\leq   E_u,
  \end{equation} 
with $E_u$  a positive constant depending on $  \underset{t_1< t < t_2} { \sup} ||u||_{2q}.$
Replacing \eqref{gradv ell 1 alt bis} in \eqref{gradv ell 1 alt}  and using 
 \eqref{eq.rtil}  we get
\begin{equation}\label{gradv ell 1 alt final}
 \begin{aligned}
  \int_{t_1}^{t_2} \! \! \! \int_{{K_R}\bigcap\{(k-u^m) >0\}} \!\!\!  u^{2({\mathfrak q}-1)}\eta^2|\nabla v|^2dx d t 
 \leq    E_u  \mu_+^{\frac{2({\mathfrak q}-1)}{m}}  \!\!
 \left( \int _{t_1}^{t_2}\Big| A_{k,R}(t)\Big|^{\frac{\tilde r}{\tilde q}}dt\right)^{\frac{2}{\tilde r}(1 +\kappa)}. 
\end{aligned}
  \end{equation} 
Inserting  \eqref{gradv ell 1 alt final} in the computations of  Lemma \ref{teo4} and  checking the validity of  Lemma \ref{for every nu},  we derive also  for the parabolic-elliptic  case that the oscillation of $u^m$ is reduced by a fixed  factor. \\
For the   $2^{nd}$ Alternative,  we observe that Lemmata \ref{log} and  \ref{t^*}      hold by replacing in the estimate of $|\nabla v|$ the constant $I_u$ with the constant $E_u$ defined in \eqref{gradv ell 1 alt bis}.
So also in this case  the oscillation is reduced. \\
Exactly as in the parabolic-parabolic case, this implies  the  H\"older continuity  of $u^m$.

%APPENDIX

%APPENDIX
%APPENDIX

%%%%%%%%%%%%%%%%%%%%%%%               APPENDIX

{\large {\bf  Appendix} }

For completeness we prove  some Lemmas used in proving  our results.
 \appendix

\section{\normalsize { Proof of Lemma \ref{EnEst}} }
\label{cap:Proof of Lemma}

Let us prove  \eqref{EE}
starting from the definition of weak solution  with the Steklov averages \eqref{weak steklov} and with $\psi=-(k-u^m)\eta^2$. Integrating from $t_1$ to $t_2$  and letting $h\to 0$, we obtain
\begin{equation}
\label{weak}
\begin{aligned}
& - \int_{\tilde Q} u_t (k-u^m)\eta^2  dx dt + \int_{\tilde Q}   \nabla u^m \cdot \nabla (-(k-u^m)\eta^2 )dxdt\\
& =  \int_{\tilde Q} \! \!
u^{{\mathfrak q}-1}  \nabla v \cdot  \nabla( -(k-u^m)\eta^2 \	 dx dt 
+   \int_{\tilde Q}\! B(-(k-u^m))\eta^2 dx dt,
\end{aligned}
\end{equation}
 where we have denoted by $\tilde Q:= K_R\times [t_1,t_2].$
 We rewrite \eqref{weak} as $ M_1 +M_2 =  M_3  + M_4$. \\
 For $k>u^m$  the following identity holds
 \begin{equation}\label{eq.id}
 \begin{aligned}
 - \! \int_{K_R}(k-u^m) \eta^2 u_t\, dx =\frac{1}{m} \int_{K_R}\frac{d}{dt} \left( \int_0^{k-u^m}(k-\xi)^{\frac{1}{m}-1}\xi d\xi  \right)  \eta^2  dx. \\
\end{aligned}
\end{equation} 
Setting  $\Phi(\xi) =: (k-\xi)^{\frac{1}{m}-1}\xi, $  integrating by parts  the right side of \eqref{eq.id} with respect to $t$   leads to
\begin{equation*}
 \begin{aligned}
& M_1 =
 \frac{1}{m}\int_{K_R\times t_2} \!\! \! \! \Big( \int_0^{k-u^m} \! \! \! \Phi(\xi) d\xi  \Big) \eta^2  dx 
-\frac{1}{m}\int_{K_R\times t_1}\!\!\! \Big( \int_0^{k-u^m}\Phi(\xi)  d\xi  \Big) \eta^2  dx\\
& -\frac{2}{m} \int_{\tilde Q}\Big( \int_0^{k-u^m} \!\Phi(\xi)  d\xi  \Big)\eta \ \eta_t \ dx dt=
M_{11} - M_{12} - M_{13}.
\end{aligned}
\end{equation*}   
 By standard  calculations we derive
\begin{equation*} 
 \begin{aligned}
 & M_2 
 =\int_{\tilde Q}|\nabla((k-u^m) \eta)|^2 dx dt -\int_{\tilde Q}(k-u^m)^2 |\nabla \eta|^2 dx dt =
  M_{21} -M_{22},
\end{aligned}
  \end{equation*}    
 \begin{equation*}
 \begin{aligned}
 &M_3 \!\!
=\int_{\tilde Q} \! -\eta \nabla((k-u^m)\eta)(u^{{\mathfrak q}-1}\nabla v)
  - \!\! \int_{\tilde Q}(k-u^m) \eta \nabla\eta (u^{{\mathfrak q}-1}\nabla v) =
  M_{31} + M_{32}
\end{aligned}
  \end{equation*} 
  By using Young inequality  we have   
  \begin{equation*}
 \begin{aligned}
 & M_{31} \leq 
  \frac{1}{2}\int_{\tilde Q}|\nabla(k-u^m)\eta|^2dx dt +\\
  & \frac{1}{2}\int_{t_1}^{t_2}\int_{{K_R}\cap\{(k-u^m) >0\}}u^{2({\mathfrak q}-1)}\eta^2|\nabla v|^2dxdt =
  M_{311} + M_{312}
 \end{aligned}
  \end{equation*}  
  
  In the same manner we obtain
 \begin{equation*}
 \begin{aligned}
 &M_{32} \leq
 \frac{1}{2}\int_{\tilde Q}(k-u^m)^2|\nabla \eta|^2 +\\
 &\frac{1}{2} \int_{t_1}^{t_2}  \int_{{K_R}\cap\{(k-u^m)>0\}}u^{2({\mathfrak q}-1)}\eta^2|\nabla v|^2 dxdt =
 M_{321}+  M_{322}.
\end{aligned}
\end{equation*}   

With $ A_{k,R} (t)= \{x \in K_R: (k-u^m(x,t))>0 \} $, applying   the  H\"older inequality first in the variable $x$ and  then in $t$   we obtain
\begin{equation}
\label{m322}	
\begin{aligned}
&M_{312}+M_{322}\leq
\mu_+^{\frac{2({\mathfrak q}-1)}{m}}\! \Big( \int_{t_1}^{t_2} \left(\! \int_{K_R } \!|\nabla v|^{2q}dx\Big)^{\frac{r}{q}}dt\right)^{\frac{1}{r}}\!\! \Big(\! \int_{t_1}^{t_2}\Big|A_{k,R} (t) \Big|^{\frac{r(q-1)} {q(r-1)} }dt\Big) ^{\frac{r-1}{r}}
\end{aligned}  
\end{equation}

%%%%%%%%%%%%
We first  observe that,
  Lemma \ref{Heat} is satisfied with $ w = u $, $p= 2q, \ p'=2,  \frac{1}{2}-\frac{1}{2q }= \frac{\kappa}{N} \ , 
 q =\frac{N}{N-2\kappa}, \kappa$ in \eqref{eq.rtil}, then
  for $(t_2-t_1)$ small enough  there exists a constant $I_u>0$ such that
\begin{equation}
\label{gradv}	
\begin{aligned}
& \Big(\int_{t_1}^{t_2} \Big(\int_{K_R}|\nabla v|^{2q}dx\Big)^{\frac{r}{q}}dt\Big)^{\frac{1}{r}}\\
 \leq
&2^{1- {\frac{1}{r}}} \left[\! \bar C_0 \underset{t_1 <t < t_2}{\text{ sup}}\Big(\! \int_{K_R} \!|u(x,t)|^{2}dx\Big)^{\frac{1}{2}}\! +\!\Big(\! \int_{K_R}|\nabla v_0(x)|^{2q}\Big)^{\frac{1}{2q}} \right] ^{2} \! \!\!(t_2-t_1)^{\frac{1}{r}} \!=:I_u.
\end{aligned}  
\end{equation}
 By inserting \eqref{gradv} in \eqref{m322} and using  \eqref{eq.rtil} it follows
\begin{equation}
\label{MM}
\begin{aligned}
&M_{312}+M_{322}
 & \leq  I_u  \  \mu_+^{\frac {2(q-1) } {m } } \ \left( \int _{t_1}^{t_2}\Big| A_{k,R}(t)\Big|^{\frac{\tilde r}{\tilde q}}dt\right)^{\frac{2}{\tilde r}(1 +\kappa)}.
\end{aligned}  
\end{equation}
Note that\,\, $ I_u$\,\, is a constant depending on
$ || u||_{L^\infty(t_1,t_2)L^2(K_R))
}  $ and  $||\nabla v_0(x)||_{L^{2q}(K_R)}$.
%%%%   
 In order to estimate the term  \,$ M_4 $\, we recall that, since  we are using the truncation \, $ (k -u^m) $,\, we take \,$k = \mu^- + \frac{\omega}{2^{s_0}}.$ 
By assumptions on  $ B $\, and taking into  account that\, $(k -u^m) \leq \frac{\omega}{2^{s_0}} < 1 $
  \begin{equation*}
 M_4 \leq  C \! \!  \int_{\tilde Q} |\nabla(k- u^m)|^2 \eta^2 \ dx dt +  \int_{\tilde Q}\! \!  \phi(k-u^m)\eta^2dx dt=
 M_{41} + M_{42}
 \end{equation*} 
 An application of  the Young Inequality leads,
 for a suitable \,$\gamma >1,$ to
\begin{equation*}
 \begin{aligned}   
&M_{41}= C \int_{\tilde Q} |\nabla(k- u^m)|^2 \eta^2 dx dt \leq\\
 &\frac{C}{2}  \int_{\tilde Q}|\nabla(k-u^m)\eta|^2 + C \int_{\tilde Q}(k-u^m)^2|\nabla \eta|^2dx dt =
  M_{411}+M_{412}. 
\end{aligned}
 \end{equation*}  
 \begin{equation*}
 \begin{aligned}
 M_{42}= \!\! \int_{\tilde Q} \!\phi(k-u^m)\eta^2dx dt \leq
 || \phi||_{L^{q,r}(\tilde Q) }  \!
 \left( \int _{t_1}^{t_2}\Big| A_{k,R}(t)\Big|^{\frac{\tilde r}{\tilde q}}dt\right)^{\frac{2}{\tilde r}(1 +\kappa)} \!\!.
 \end{aligned}
 \end{equation*}    
 We observe that,  choosing \,$C$\, small enough, we have for a suitable constant \,$ \gamma^*>0$\,
  \begin{equation*}
 \begin{aligned}
& M_{21} -M_{311}-M_{411}\geq
 \gamma^*  \int_{\tilde Q}|\nabla(k-u^m)\eta|^2dx dt.
\end{aligned}
 \end{equation*} 
 
Collecting all the previous inequalities obtained for $M_i, \ i=1,...4,$
   Lemma \ref{EnEst} is proved.\\

%%%%%%%%%%                       
%{ \bf  A 2. \ Proof of Lemma \ref{be-ab}}\\

\section{ \normalsize{Proof of Lemma \ref{teo4}}} 
\label{cap:Estimate Phi}
First we estimate
$\int_0^{ (k-u^m\!)}  \Phi(\xi) \ d\xi$.\\
Let $\Phi(\xi) =:
  (k-\xi)^{\frac 1m-1 }\xi $
 and $\alpha= 1-\frac1m$. 
 We prove that there exist two  positive constants $\widehat c \leq \frac12,$ \,\,and \,\, $\check{c} \geq  \frac{ m^2}{1+m},$  such that
\begin{equation}
\label{ab2}
\begin{aligned}
 \widehat c \   \frac {(k-u^m)^2}{(\theta_0 + \mu_-)^{\alpha}} \ \ 
\leq \int_0^{ (k-u^m)} \Phi(\xi) \ d\xi  \ \leq  \  \check{c}  \  \frac {  (k-u^m)^2}{( \theta_0 +\mu_-)^{\alpha}}.
\end{aligned}
\end{equation}

 For $0<u^m<k, $  we  derive
\begin{equation*}
\begin{aligned}
& \int_0^{ (k-u^m)} \! \! \Phi(\xi) \ d\xi = \!  \int_0^{ (k-u^m)} [ k(k-\xi)^{- \alpha }
-(k-\xi)^{1- \alpha }] d \xi
\\
&= k \frac { k^{1- \alpha }- (u^m)^{1- \alpha } }{1-\alpha } - \frac{ k^{2- \alpha }- (u^m)^{2- \alpha } } {2-\alpha }.
\end{aligned}
 \end{equation*} 

We introduce the  function 
\begin{equation*}
\begin{aligned}
r(u^m) =: k \frac { k^{1- \alpha }- (u^m)^{1- \alpha } }{1-\alpha } - \frac{ k^{2- \alpha }- (u^m)^{2- \alpha } } {2-\alpha }
 - \widehat c \   \frac {(k-u^m)^2}{(\theta_0 + \mu_-)^{\alpha}},
\end{aligned}
 \end{equation*}
with $r(0) =  (k)^{2- \alpha} \Big( \frac 1{(1-\alpha)(2-\alpha) } - \widehat c \Big)$ and $r(0)>0$, if $\widehat c< \frac 1{(1-\alpha)(2-\alpha) }$. Moreover $r(k)=0$.
 Now we observe that $r'(u^m)<0$ if $\widehat c\leq \frac12$. As a consequence, the function $r(u^m)$, initially positive, is decreasing up to $(k,0)$ and always non negative in
$0<u^m<k$. 
%%%%%%%%%%%%%%%%%%%%%%%%        
For the inequality on the right of \eqref{ab2}
following the previous steps, we define the function
\begin{equation*}
\begin{aligned}
g(u^m) =: k \ \frac { (k)^{1- \alpha }- (u^m)^{1- \alpha } }{1-\alpha } - \frac{ k^{2- \alpha }- (u^m)^{2- \alpha } } {2-\alpha }
 - \check{c} \   \frac {(k-u^m)^2}{(\theta_0 + \mu_-)^{\alpha}}
\end{aligned}
 \end{equation*}
and $g(k)=0$ and $g(0) < 0$  if  $\check{c} >\frac 1{(1-\alpha)(2-\alpha) }$. Moreover $g$ is decreasing up to its minimum reached at $u^m= \frac{\theta_0+\mu_-} {(2\check{c})^{ }1/\alpha}.
$ Then $g$ is negative and  \eqref {ab2} is proved.\\
Now we are ready to prove Lemma \ref{teo4}.

Let us introduce some technical tools. By translation we assume $(0,\bar t)=(0,0)$.

 Let us  construct a family of nested  cylinders  $Q_{R_i}(\hat \theta)$ with $ R_i= \frac R2+\frac R {2^{i+1}  }, \ i=0,1,2,...$ for which the assumptions of  Lemma \ref{teo4} hold.  After a translation we assume $(0,\bar t)=(0,0)$.
Let $\zeta_i$ be a piecewise cutoff function in $Q_{R_i}(\hat \theta)$ such that
\begin{equation}\label{eq.2}
\left\{ \begin{array}{l}
\begin{aligned}
 &0<\zeta_i (x,t)<1 , \quad  (x,t)\in Q_{R_i}(\hat \theta), \quad 
 \zeta_i=1 , \quad  (x,t)\in Q_{R_{i+1}}(\hat \theta), \\
 &\zeta_i=0, \quad {\rm on \ the \  parabolic \ boundary \ of} \ Q_{R_i}(\hat \theta);\\
 &|\nabla \zeta_i|\leq \frac { 2^{i+1}} { R},  \quad  \ 0\leq (\zeta_i)_t \leq \frac { 2^{2(i+1)}} {  \theta_0^{- \alpha}\ R^2}.
 \end{aligned}
 \end{array} \right.
\end{equation} 
Apply the Energy Estimates over the cylinders $Q_{R_i}(\hat \theta)$ to the truncated  functions $(k_i-u^m)$ 
with 
$k_i= \displaystyle \mu_- + \frac \omega{2^{s_0+1} } + \frac \omega{2^{s_0+1+i} },  \  \ i=0,1,2,...$.
In \eqref{EE} the first term on the left and the second on the right  include the term
$
 \int_0^{ (k-u^m)} \Phi(\xi) \ d\xi, \quad   {\rm with} \ \Phi(\xi)= (k-\xi)^{\frac 1m -1}\xi,
 $ 
which estimate
 from below and  from  above is presented in  Appendix \ref{cap:Estimate Phi}. 
 With $k$ replaced by $k_i,$  by using \eqref{ab2}, we have 
\begin{equation}\label{3Phi}
\int_{K_ {R_i }\times \{t_2\}} \! \! \zeta_i^2 \Big[ \int_0^{ (k_i-u^m)} (k_i-\xi)^{\frac 1m -1 } \xi \ d\xi  \Big] dx 
\geq 
 \widehat c \  \   \frac{ ||\zeta_i(k_i-u^m)||_{L^2 (K_{R_i})}^2}{  (\mu_- +\theta_0)^{\alpha}}.
\end{equation} 
In \eqref{3Phi} we take the sup in time since $t_2$ is arbitrary.
 From  \eqref{EE}  multiplied by  $k_0^{\alpha}=(\mu_- +\theta_0)^{\alpha} $ and using  \eqref {3Phi}, we get
\begin{equation}\label{EE2}
\begin{aligned}
&\underset{-\theta_{0}^\alpha R_i^2 <t<0} \sup  || \zeta_i(k_i-u^m)||_{L^2(K_ {R_i })}^2    + k_0^{\alpha} 
|| |\nabla (\zeta_i (k_i-u^m))||_{L^2(R_i )}^2 \\
& \leq {\bf C}m  k_0^{\alpha} \Big\{ \int^0_{-\theta_0^{-\alpha}R_i^2}\!\!  \int_{K_{R_i}} \!\! (k_i-u^m)^2 |\nabla \zeta_i|^2 \ dx dt \\
&+\int^0_{-\theta_0^{-\alpha}R_i^2}\!\! \int_{K_{R_i}} \!\! \ \Big[ k_0^{-\alpha}  (k_i-u^m)^2
  \Big] |\zeta_i (\zeta_i)_t| dx dt \\
& +\Big(\!\!  \int^0_{-\theta_0^{-\alpha}R_i^2}  | A_{k_i,R_i}(t)|^{\frac {\tilde r } {\tilde q }}  dt \Big)^{\frac {2 } {\tilde r } (1+\kappa)} \Big\} \\
& \leq  {\bf C} m  k_0^{\alpha}  \Big\{ (1+m) \!
 \frac { 2^{2(i+1)} } {R_i^2}\!\!  \frac{\omega^2}{2^{2s_0}}\!\! \ \int_{-\theta_{0}^\alpha R_i^2 }^{0}| A_{k_i,R_i}|(t) dt \\
  &+ \Big(\!\! \ \int^0_{-\theta_0^{-\alpha}R_i^2}  | A_{k_i,R_i}(t)|^{\frac {\tilde r } {\tilde q }}  dt \Big)^{\frac {2 } {\tilde r } (1+\kappa)}\!\! \ \Big\},
  \end{aligned}
\end{equation} 
with $  A_{k_i,R_i}(t)= \{ x \in K_{R_i} : u^m<k_i  \} $.
In order to simplify the computations we perform the following change of time variable in \eqref{EE2}: \ $ z=\theta_0^{\alpha}t$. As a consequence we have
$Q_{R_i}(\hat \theta) \rightarrow  Q_i=Q(R_i,R_i^2); \ 
 u(x,t)  \rightarrow  u(x, \theta_0^{-\alpha} z)={\rm v}(x,z) ; \  \zeta(x,t) \rightarrow  \ \hat \zeta (x,z);  \ \  dt=  \theta_0^{-\alpha} dz.$
We obtain
\begin{equation}
\label{C}
\begin{aligned}
& \underset{(-R_i^2,0)} {\text{ess\,sup}}  || \hat{ \zeta_i}(k_i-{\rm v}^m)||_{L^2(K_ {R_i })}^2    +  
|| \nabla (\hat{\zeta_i} (k_i-{\rm v}^m))||_{L^2(Q_i )}^2 \\
& \leq {\bf C}  m \Bigg\{\! (1+m)  \frac{\omega^2 } {2^{2s_0} } 
 \frac { 2^{2(i+1)} } {R_i^2}  \frac {k_0^\alpha} { \theta_0^\alpha} |Z_i| + \!
 \frac { k_0^{\alpha}  }{ \theta_0^{\alpha(1- \frac 1r)} }
  \Bigg(\! \int^0_{-R_i^2} \! | Z_i(z)|^{\frac {\tilde r} {\tilde q}} dz  \Bigg)^{\frac {2 } {\tilde r } (1+\kappa)} \! \Bigg\},
 \end{aligned}
\end{equation} 
%%%  
with $Z_i(z)=  \{ x\in K_{R_i}, \ :  {\rm v}^m(x,z)<k_i \}$ and
 $|Z_i|= \int^0_{-R_i^2}  | Z_i(z)| dz.  $
 
 %%%
 The sequence  $(Z_i)$  is connected with two sequences $(X_i)$ and $ (Y_i)$  so defined:                           
\begin{equation}
\label{defXY}
X_{i} =  \frac{|Z_i| } {|Q_i| } ;  \quad \quad 
Y_{i} = \frac{ \Bigg( \int^0_{-R_i^2}  | Z_i(z)|^{\frac {\tilde r} {\tilde q}}dz \Bigg)^{\frac {2 } {\tilde r }  }} {|K_{R_{i}}| },  
\end{equation} 
which satisfy Lemma \ref{fastconv}.
In fact we prove that
$ X_{i+1} \leq c \ 16^i (X_i^{ 1+\hat \alpha}+ X_i^{\hat\alpha} Y_i^{ 1+\kappa}),
 \ \hat\alpha= \frac 2{ N+2}.$
By  inequality \eqref{sob2}  in  the Sobolev Lemma \ref{sobolev} applied to $(k_i-{\rm v}^m)$ we have
\begin{equation}
\label{Sob1}
\begin{aligned}
& || \hat{ \zeta_i}(k_i-{\rm v}^m)||_{L^2({Q_i })}^2 
  \leq {\bf C} m(\frac{k} {\theta_0})^\alpha   \Bigg\{ \!(1  \! +  \! m)  \frac{\omega^2 \ 2^{2(i+1)} } { 2^{2s_0} R_i^2}     |Z_i|^{1+ \frac 1{N+2} } \\ 
 &+ | Z_i|^{\frac {2} {N+2}}
  \theta_0^{\frac{\alpha}{r} }   \Bigg(\! \int^0_{-R_i^2}  | Z_i(z)|^{\frac {\tilde r} {\tilde q}} dz \!  \Bigg)^{\frac {2 } {\tilde r } (1+\kappa)} \Bigg\} .\\
\end{aligned}
\end{equation} 
Moreover, we have 
\begin{equation}
\label{Sob3}
\begin{aligned}
& \int_{Q_i} \! |\hat  \zeta_i(k_i-{\rm v}^m)|^2 dx dz \\
&\geq \! |k_{i+1}-k_i|^2  \! \! \int^0_{-R_i^2} \! \! 
 | \{ (x,z)\in R_{i+1}: \! {\rm v}^m<k_{i+1} \}| dz
 = \! \Big( \frac {\omega} {2^{s_0+i+2  } }\Big)^2 |Z_{i+1}|,
 \end{aligned}
\end{equation}
from \eqref{Sob1} and  \eqref{Sob3} we derive the following  upper bound of $ |Z_{i+1}|$:
\begin{equation}
\label{Zi+1}
\begin{aligned}
&|Z_{i+1}| \leq  {\bf C} \ m (\frac{ k} {\theta_0})^\alpha  
 \Bigg\{  (1+m) \   \frac { 2^{4(i+2)}} { R_i^2}   |Z_i|^{1+ \frac 2{N+2} } \\
&
+ ( \frac {\omega} {2^{s_0+i+2  } })^{-2} | Z_i|^{\frac {2} {N+2}}
\theta_0^{\frac{\alpha}{r} }   \Bigg( \int^0_{-R_i^2}  | Z_i(z)|^{\frac {\tilde r} {\tilde q}} dz   \Bigg)^{\frac {2 } {\tilde r } (1+\kappa)}\Bigg\}\\
&  
\leq \! c \ 16^i \Bigg \{\frac { 1} { R_i^2}  \! |Z_i|^{1+ \frac 2{N+2}} 
 + \theta_0^{\alpha (\frac1{r}-\frac2{\alpha})} \! | Z_i|^{\frac {2} {N+2}}
 \Big(\! \int^0_{-R_i^2}  | Z_i(z)|^{\frac {\tilde r} {\tilde q}} dz   \Big)^{\frac {2 } {\tilde r } (1+\kappa)} \! \! \Bigg\},
\end{aligned}
\end{equation} 
where the  constant $c$ depends on the data and $r$ and $\tilde r$ in \eqref{eq.rtil}.
Divide    \eqref{Zi+1}    by $|Q_{i+1}| $: in the first term  on the right  we have
$
\frac{|Z_{i}|^{1+ \frac 2{N+2}}  } {R_i^2 R_i^{N+2}}= \frac{|Z_{i}|  } { R_i^{N+2}} \times \frac{|Z_{i}|^{ \frac 2{N+2}}  } {R_i^{(N+2) \frac{2}{N+2}} }
$
and in the second term 
$$
\frac{  | Z_i|^{\frac {2} {N+2}}
 \Big( \int^0_{-R_i^2}  | Z_i(z)|^{\frac {\tilde r} {\tilde q}} dz   \Big)^{\frac {2 } {\tilde r } (1+\kappa)} } { R_i^{N}\times  R_i^{2}}= \frac{|Z_{i}|^{ \frac 2{N+2}}  } {R_i^{(N+2) \frac{2}{N+2}} }\times  
 R_i^{N\kappa}  \ 
 \frac{ \Big( \int^0_{-R_i^2}  | Z_i(z)|^{\frac {\tilde r} {\tilde q}}dz \Big)^{\frac {2 } {\tilde r }(1+\kappa)  }} {(R_i^N)^{1+\kappa} } 
$$
Then
\begin{equation}
\label{Xineq}
X_{i+1}
\leq c \ 16^i\Big\{ X_i^{1+\frac{2}{n+2}} +
\frac{1}{\theta_0^{\alpha(\frac{2}{\alpha}-\frac{1}{r})}}\,R_{i+1}^{(N\kappa)}X_i^{\frac{2}{N+2}}Y_i^{1+\kappa}\Big\}
\end{equation} 
%%%

 Recalling that \,\,$ \frac{1}{\theta_0^\alpha}< \frac{1}{R^\epsilon} $\,\, we have\,\, $\frac{1}{\theta_0^{\alpha(\frac{2}{\alpha}-\frac{1}{r})}}< \frac{1}{R^{\epsilon\alpha(\frac{2}{\alpha}-\frac{1}{r})}},$ 
 choosing  $\epsilon $ such that 
 $\epsilon  \alpha(\frac{2}{\alpha}-\frac{1}{r})  < Nk,$
 we obtain
 \begin{equation}
 X_{i+1}\leq c \  16^i \bigg\{ X_i^{1+\frac{2}{N+2}} + X_i^{\frac{2}{N+2}}Y_i^{1+\kappa}\bigg \},
 \end{equation}
 where \,\,$ c$\,\, is a constant independent  of $R,\ \omega, \ A$.\\
 %%%%%%%%%%%%%
Now we prove that   \quad $  Y_{i+1} \leq c \ 16^i (X_i+ Y_i^{ 1+\kappa})$.
\begin{equation*}
\begin{aligned}
 & Y_{i+1}(k_{i}- k_{i+1} )^2=
 \frac{1}{|K_{R_{i+1}}|} \! \Bigg[\! \int_{-R_{i+1}^2}^0 \!\!  (k_{i}-k_{i+1})^{\tilde{r}} \Bigg(\! \int_{ \{ x \in{ R_{i+1}:  {\rm v}^m(x,z)<k_{i+1} \} }}\! dx \!\Bigg)^{\frac{\tilde{r}}{\tilde{q}}} \Bigg]^{\frac{2}{\tilde{r}}}\\
 &\leq
 \frac{1}{|K_{R_{i+1}}|}\Bigg[ \int_{-R_{i+1}^2}^0   \Bigg( \int_{  R_{i+1}}(\hat \zeta_i(k_i -{\rm v}^m))^{\tilde{q}}\,\, dx\Bigg)^{\frac{\tilde{r}}{\tilde{q}}} dz \Bigg]^{\frac{2}{\tilde{r}}}.
 \end{aligned}
\end{equation*}
Then (from $k_{i+1}$ to $k_i$)
\begin{equation}
Y_{i+1} |k_{i+1}-k_i|^2 \leq 
 |K_{ R_{i+1}}|^{-1} \  || \hat \zeta_i( k_{i}-{\rm v}^m)  ||^2_{\tilde q,\tilde r;Q_{i}}.
\end{equation}
Moreover,  we apply \eqref{embed}  in Lemma \ref{embeded} 
%%%%%%%%%%%%%%%%
 to the truncated $\hat \zeta_i (k_i-{\rm v}^m)$, which is zero on the boundary from the definition \eqref{eq.2} of $\hat \zeta_i$.
Then  from the inequality \eqref{C}
 we have 
\begin{equation}
\begin{aligned}
\label{Y}
&Y_{i+1} |k_{i+1}-k_i)|^2 \leq  |K_{ R_{i+1}}|^{-1}  ||\hat  \zeta_i( k_{i}-{\rm v}^m)  ||^2_{\tilde q,\tilde r;Q_{i}} 
\leq \gamma |K_{ R_{i+1}}|^{-1} ||\hat \zeta_i( k_{i}-{\rm v}^m) ||^2_{V^2(Q_i)} \\
& \leq   |K_{ R_{i+1}}|^{-1} \Big(c_1 |Z_i|+ c_2  \Big( \int^0_{-R_i^2}  | Z_i(z)|^{\frac {\tilde r} {\tilde q}} dz \Big)^{\frac {2 } {\tilde r } (1+\kappa)}\Big) ,
 \end{aligned}
\end{equation}
with $c_1=Cm \gamma   (1+m){ (\frac{\omega } { 2^{s_0 } })^2}
 \  \frac { 2^{2(i+1)}} { R_i^2} , \ c_2=Cm \gamma  \theta_0^{\frac{\hat\alpha}{r}}.$
Then  from \eqref{Y}
\begin{equation}
\begin{aligned}
&Y_{i+1} \leq \tilde c_1 \Bigg\{\frac{\omega^2 2^{2(i+1)} } { 2^{2s_0 }R_i^2  }  \frac 1{\frac {\omega^2} { 2^{2(s_0+i+2 )}} } |K_{ R_{i+1}}|^{-1} |Z_i| \\
&+ 
 \frac{ (\frac{\omega } { 2^{s_0  }})^{\frac {\hat\alpha}{q}}} {( \frac{\omega } { 2^{s_0  }})^2 }  ( 2^{2(i+2 )} ) |K_{ R_{i+1}}|^{-1} \Bigg( \int^0_{-R_i^2}  | Z_i(z)|^{\frac {\tilde r} {\tilde q}} dz \Bigg)^{\frac {2 } {\tilde r } (1+\kappa)} 
  \Bigg \}.
   \end{aligned}
\end{equation}
  We conclude that
$
Y_{i+1} \leq \bar c_1 \ 16^i (X_i+ Y_i^{ 1+\kappa}) .
$
The sequences $ X_i, \ Y_i$ satisfy Lemma \ref{fastconv} with $\rm b=2^{-4}$ and
$$
 X_0+Y_0^{1+k} \leq ( \frac 1{2 \bar c} )^{ \frac{1+k } {\sigma}} (\frac1{16})^{ \frac{1+k } {\sigma^2}}  := \nu<1.
$$
Then $X_i, Y_i \to 0$ as $i \to \infty.$ With such  $\nu$ the hypothesis \eqref{eq.5} of  Lemma \ref{teo4}  holds, then as a consequence, returning to the original coordinate in time, we obtain 
\eqref{eq.1}.

%%%%%%%%%%%%%%%%%%%%%%%%%%%%  

\section{\normalsize{Proof of Lemma \ref{for every nu}  in first Alternative}}
\label{cap:forevery}

First we estimate   $\left( \int _{t^*}^{0}\Big| A_{k,R}(t)\Big|^{\frac{\tilde r}{\tilde q}}dt\right)^{\frac{2}{\tilde r}(1 +\kappa)}.$

 Since for any $ t\in \left[ -t^*,0\right],\ \Big| A_{k,R}(t)\Big|\,\leq |K_R|$ 
we have
\begin{equation}
\begin{aligned}
\label{eq.A^+}
& \left(\! \int _{-t^*}^{0}\Big| A_{k,R}(t)\Big|^{\frac{\tilde r}{\tilde q}}dt\right)^{\frac{2}{\tilde{r}}(1 +\kappa)} 
\leq 
\Big(t^* \left(|K_R|\right)  ^{\tilde r/\tilde q  }\Big)^{ \frac {2(1+ \kappa)} {\tilde r}} \\
&\leq
\Big( |K_R|^{\tilde r/\tilde q  }\Big)^{ \frac {2(1+ \kappa)} {\tilde r}}\Big(\frac{R^2 } {\theta_0^{ \alpha} } \Big)^ {\frac {2(1+ \kappa)} {\tilde r}}
=|K_R|\,R^{N\kappa}\frac{1}{\theta_0^{\alpha(\frac{2(1+k}{\tilde r})}},
\end{aligned}
\end{equation}
where  $\kappa$ satisfies \eqref{eq.rtil}.\\

Moreover, we have 
$
R^{N\kappa}  \theta_0^{-\alpha \frac{2(1+\kappa)}{\tilde r }} < R^{Nk-\varepsilon \frac{2(1+\kappa)}{\tilde r }}.$
Pick \,$ \varepsilon$ so small in order to have\quad$\frac{N\kappa}{\frac{2(1+\kappa)}{\tilde r }}  >2  \varepsilon$. Then
$
N\kappa-\varepsilon \frac{2(1+\kappa)}{\tilde r }  > \varepsilon
 \frac{2(1+\kappa)}{\tilde r }$
and taking into account  that $
(\frac{\omega}{A})^\alpha>  R^\varepsilon,
$
it follows
$  R^{Nk-\varepsilon \frac{2(1+\kappa)}{\tilde r }} < \big(\frac{\omega}{A}\big)^{\alpha \frac{2(1+\kappa)}{\tilde r }}$.
Insert the  last estimate in  \eqref{eq.A^+} to get 
\begin{equation}
\label{3eq.A^+}
 \left(\! \int _{-t^*}^{0}\Big| A_{k,R}(t)\Big|^{\frac{\tilde r}{\tilde q}}dt\right)^{\frac{2}{\tilde{r}}(1 +\kappa)} 
\leq |K_R| \     \Big(\frac{\omega}{A}\Big)^{\alpha  \frac{2(1+k)}{\tilde r}}.
\end{equation}

%{ \bf A5 - Proof of Lemma \ref{for every nu}  in first Alternative}\\
%%%%%%%      
Now we prove Lemma  \ref{for every nu}.
Let  $s_0$ be  the smallest positive integer  such that \,$ \frac{\omega}{2^{s_0}}\leq {1}$. Let \,$s_1 > s_0 +2$\,  an integer to be stated later.\,
  By Lemma \ref{teo4}, in  \,\,$ K_{\frac{R}{2}}\times (\bar{t} -\hat{\theta},\bar{t} )$, \,\,$ u  > (\mu_- +\frac{\omega}{2^{s_0 +1}})^{\frac{1}{m}}$\,a.e. \\
   Set\, $\hat{k} = (\mu_- +\frac{\omega}{2^{s_0 +1}})^{\frac{1}{m}},\,\, 
   \hat{c}=( \frac{\omega}{2^{s_1}})^{\frac{1}{m}}$.\,\, Define in\,$K_{\frac{R}{2}}\times (\bar{t} -\hat{\theta},0) $ 
\begin{equation}
\hat{\varphi}(u)= ln^+\left( \frac{\hat{H}}{\hat{H} -(\hat{k} -u)_+ + \hat{c}}\right), 
\end{equation}
where \,\,$ \hat{H} =\underset{K_{\frac{R}{2}}\times (\bar{t} -\hat{\theta},0)} {\text{ ess\,sup}}(\hat{k}-u)_+$.\,
 \,\, We have\,\,$\hat{\varphi} (u)\leq ln^+\left( \frac {\hat{H}}{ \hat{c}}\right).$\,\,

 Since 
 $\hat{H}\leq (\mu_- +\frac{\omega}{2^{s_0 +1}} )^{\frac{1}{m}} -\mu_-^{\frac{1}{m}}$,   neglecting $\mu_-^{\frac{1}{m}} $ and  observing that $f(\mu_-)=:(\mu_- \frac{\omega}{2^{s_0 +1}} )^{\frac{1}{m}} -\mu_-^{\frac{1}{m}}$ attains its max value at   $\mu_-=0$, we have
\begin{equation}
\label{log2}
 \hat{\varphi}(u) \leq \frac 1m (s_1-s_0) log 2.
\end{equation}

 Moreover,
  $\hat{\varphi}'(u)= \frac{\partial \hat{\varphi}(u) } { \partial u}= \frac{ -1}{ \hat{H} -(\hat{k} -u)_+ +\hat{c}}$, \ $\hat{\varphi}''(u)= \frac{\partial^2 \hat{\varphi}(u) }{ \partial u^2} = (\hat{\varphi}'(u))^2$ \,\, and by Lemma \ref{teo4}\,\, at the level time $ \bar{t}-\hat{\theta}  $, $\hat{\varphi}(u) = 0$\, in \,$K_{\frac R2}\times\{\bar{t}-\hat{\theta}\}.$\, 
  To prove  Lemma \ref{for every nu}  we consider the definition of weak solution with the Steklov average. 
 Pick as test function $ \psi=(\hat{\varphi}^2(u_h))' \xi^2, $ where $\xi = \xi(x)$ is a cut-off  function with $\xi = 1  \  \text{in}\, K_{\frac R4}, \ \xi =0 \ \text{on}\,  \ \partial K_{ \frac R2}, |\nabla \xi| \leq \frac 8 {R  }$\,\,For any \,$\bar{t}-\hat{\theta}\leq t\leq 0$\,\, understanding  \,$h\rightarrow \,0$, we directly compute 
  \begin{equation}
  \label{h}
\begin{aligned} 
&\iint_{K_{\frac R4 } \times (\bar{t}-\hat {\theta},t)}(u )_t (\hat{\varphi}^2(u))'\xi^2 dx d\tau\  \ = \\
&-\iint_{ {K_{\frac R4 } \times (\bar{t}-\hat {\theta},t)}}(\nabla u^m)\cdot\nabla\left(  (\hat{\varphi}^2((u)' \xi^2 \right)dxd\tau\\
&- \iint_{K_{\frac R4 } \times (\bar{t}-\hat {\theta},t)}(u^{q-1}\nabla v)\cdot\nabla\left(  (\hat{\varphi}^2((u))' \xi^2  \right)dxd\tau \\
&+ \iint_{ { {K_{\frac R4 } \times (\bar{t}-\hat {\theta},t)}}} (C(|\nabla u^m)|^2 +\phi)(\hat{\varphi}^2(u))'\xi^2dxd\tau\\
&
:=-\int_{\bar{t}-\hat\theta}^{0}\hat{I}dt\,+\int_{\bar{t}-\hat\theta}^{0} \hat{J} dt + \int_{\bar{t}-\hat\theta}^{0}\hat{L}_1 dt+ \int_{\bar{t}-\hat\theta}^{0} \hat{L}_2 dt.
\end{aligned}
\end{equation}
 For the first term on the left hand side
\begin{equation}
\label{h1}
\begin{aligned}
&\iint_{K_{\frac R4 } \times (\bar{t}-\hat {\theta},t)}((\hat{\varphi}(u) \xi)^2)_t dxdt =
\int_{K_{\frac R4 } \times\{t\}}((\hat{\varphi}(u) \xi))^2 dx.\\  
 \end{aligned}
\end{equation}
 
To obtain an estimate from below of the term on the right in \eqref{h1},
let us  integrate in the smaller set
$$
\hat{P}=: \{ x\in  K_{R/4} : u(x, t)< (\mu_- + \frac {\omega}{2^{s_1} })^{\frac{1}{m}}  \},  \ \ t\in(\bar{t}-\hat\theta, 0). $$ 

In such set $\hat{P}$,  
we have
 \,$$ \hat{\varphi} (u)\geq  ln^+\frac{\hat{H}}{\hat{H}-\hat{k}+(\mu_- + \frac{\omega}{2^{s_1}})^{\frac{1}{m}}+\hat{c}}.
 $$ 
 In this inequality, we apply $ (\mu_- + \frac{\omega}{2^{s_1}})^{\frac{1}{m}} \leq (\mu_-)^{\frac{1}{m}}  + (\frac{\omega}{2^{s_1}})^{\frac{1}{m}} .   $
 
 Moreover, the right hand side is a decreasing function in  \,$\hat{H}$.  Then for\,\,$t\in (\bar{t}- \hat \theta,0) $, with $\hat H \leq \hat k - \mu_-^ {1/m},$ we have
\begin{equation}
 \hat{\varphi}(u)\geq  \ln^+\frac{\hat{k}-\mu_-^{\frac{1}{m}} }{2\hat{c}}
\end{equation} 

 Then
\begin{equation} \label{log3} 
 \int_{\hat{P}}\hat{\varphi}^2 (u)  dx \geq  \frac{\hat C}{m^2} (s_1-s_0-1)^2 \ln^22 \  |\hat{P}|,
\end{equation} 
with $\hat C$ is a constant depending on $||u||_{L^{\infty}}$.
Let us estimate the terms on the right hand side of \eqref{h}.

\begin{equation*}
 \begin{aligned}
 & -\hat{I} =    
  -\int_{K_{\frac R4}}m u^{m-1}\nabla u \cdot\left[ ( \, 2((\nabla\hat{\varphi})\hat{\varphi}' \xi^2  
  + 2\hat{\varphi}(\nabla\hat{\varphi}')\xi^2\right.\\
  &\left. +4\hat{\varphi}\ \hat{ \varphi}'\xi\nabla\xi  \right] dx \leq 
 -2m\int_{K_{\frac R4}}u^{m-1} (1+\hat{\varphi})\hat{\varphi}'^{2}\xi^2\, |\nabla u|^2dx \\ &+4m\int_{K_{\frac R4}}u^{m-1}|\nabla u||\hat{\varphi}\hat{\varphi}'|\xi\,|\nabla \xi| dx 
 =\, -\hat{I}_1 +\hat{I}_2     
  \end{aligned}
  \end{equation*}
  
Moreover, by Young inequality we have
\begin{equation*}
  \begin{aligned}
&\hat I_2 =4m\int_{K_{\frac R4}}u^{m-1}\varphi(\frac{1}{\sqrt{2}}\varphi'|\nabla u|)(\sqrt{2}\nabla\xi) \leq 4m\int_{K_{\frac R4}}u^{m-1}\varphi\left[ \frac{1}{4}\varphi'^2\xi^2\,|\nabla u|^2| +| \nabla \xi|^2\right]  dx \\
&= m\int_{K_{\frac R4}} \!\! u^{m-1} \varphi\varphi'^{2}\xi^2 |\nabla u|^2dx + 4m\int_{K_{\frac R4}}u^{m-1}  \varphi\, |\nabla \xi|^2dx =\hat I_{21} +\hat I_{22}\\
\end{aligned}
  \end{equation*} 
  
\begin{equation*}
 \begin{aligned}
& \hat J \!=\! \int_{K_{\frac R4}} \! u^{\mathfrak q-1}\nabla v \cdot\left[  2((\nabla\varphi)\varphi' \xi^2
 + 2\varphi(\nabla\varphi')\xi^2 +4\varphi \varphi'\xi\nabla\xi  \right]dx
 \\
& \leq 2\int_{K_{\frac R4}}u^{\mathfrak q-1}(1+ \varphi)\varphi'^{2} \xi^2|\nabla u||\nabla v|dx 
+4\int_{K_{\frac R4}} u^{\mathfrak q-1}\varphi|\varphi'|\xi |\nabla v||\nabla \xi|dx \\
&=\hat J_1 + \hat J_2 .
\end{aligned}
\end{equation*}
By hypothesis on $ \mathfrak q  $, we can
split\,$\mathfrak q-1$ in the sum of two positive term of the form\,
$\mathfrak q-1= \frac{m-1 }2 + \mathfrak q-\frac{m+1 }2.
$
Using again  Young's inequality  in $\hat J_1$ we obtain
 \begin{equation*}
\begin{aligned}
&\hat J_1\leq
\int_{K_{\frac R4}} \!\! (1+ \varphi)\varphi'^{2} u^{ m-1}|\xi^2|\nabla u|^2 dx+  \!\! \int_{K_{\frac R4}}  \!\! (1+\varphi)\varphi'^{2} \xi^2 u^{ 2q-m-1 } |\nabla v|^2 dx\\
&=\hat J_{11} + \hat J_{12},
\end{aligned}  
 \end{equation*}

 \begin{equation*}
\begin{aligned}
&\hat J_2 =
4\int_{K_{\frac R4}} \varphi\left[ (|\varphi'|u^{\mathfrak q-1}\xi |\nabla v|)(|\nabla \xi|)\right]dx  \leq\\
&2\int_{K_{\frac R4}} \varphi\varphi'^2 u^{2\mathfrak q-2} \xi^2 |\nabla v|^2 dx +2\int_{K_{\frac R4}}  \varphi|\nabla \xi|^2 dx
 = \hat J_{21} + \hat J_{22}
\end{aligned}  
 \end{equation*}
Note  that
$  -\hat I +\hat J +\hat L_1+\hat L_2 \leq 
  (-\hat I_{1} +\hat J_{11} +\hat I_{21}+\hat L_1)+(\hat I_{22}+\hat J_{22})$\\$ + (\hat J_{12}+ \hat J_{21}+\hat L_2) .$
\quad  Since\,$\hat L_1\leq 0 $\,
  \begin{equation*}
\begin{aligned}
&-\hat I_{1} + \hat J_{11} +\hat I_{21}+\hat  L_1 \! \leq - (m -1) \int_{K_{\frac R4}}u^{m-1} \varphi'^{2}\xi^2\, |\nabla u|^2dx   < 0 , 
\end{aligned}  
 \end{equation*}
   a negative term that can be neglected.
\begin{equation}
\label{I2final}
\begin{aligned}
&\hat I_{22}+ \hat J_{22}=
4m \int_{K_{\frac{R}{4}}} u^{m-1}\hat  \varphi |\nabla 
\xi |^2 dx +2 \int_{K_{\frac{R}{4}}}  \hat \varphi|\nabla \xi|^2 
< C_1 \int_{K_{\frac{R}{4}}} \hat  \varphi\, |\nabla \xi|^2dx\\
&\leq  C_1 (s_1-s_0-1)\ln 2 \frac {2^6 }{R^4} \  \Big( \int_{-\theta}^{0}( |A_{k, \frac R4}|^{\frac {\tilde r} {\tilde q}}dt \Big)^{\frac{2(1+k)} {\tilde r}},
\end{aligned}  
 \end{equation}
 with $ C_1=4m\mu_+^{m-1}+2.$ \\
 
%%%%%%%%%%%%%
Now we can estimate 
 \begin{equation}
 \label{J1}
\begin{aligned}
&\hat J_{12}+\hat J_{21}=\\
 &  \int_{K_{\frac R4}}\!\! (1+\hat \varphi)\hat \varphi'^{2} u^{2\mathfrak q-2}\xi^2|\nabla v|^2dx 
+2\int_{K_{\frac R4}} \!\! \hat \varphi \hat \varphi'^2 u^{2\mathfrak q-2} \xi^2 |\nabla v|^2dx \\
&< 3  \int_{K_{\frac R4}} u^{2\mathfrak q-2}(1+\hat \varphi)\hat \varphi'^2  \xi^2 |\nabla v|^2dx
\end{aligned}  
 \end{equation}
 %%%%%%%%%%
 An integration in time  of \eqref{J1}  yields
 \begin{equation}
 \label{1J1}
\begin{aligned}
&\int_{-\hat \theta }^0 ( \hat J_{12}+\hat J_{21})dt  \\
& \leq \mu_+^{\frac{ 2\mathfrak q-2} { m} } \frac{3}{m}(s_1 - s_0 )\ln2 \left( \frac{2^{s_1}}{\omega}\right)^2 \!\!  \iint_{Q_{\frac R4}(\hat \theta)} \!\!\xi^2 |\nabla v|^2dx dt,   \\
\end{aligned}  
\end{equation}
 where the estimate  $ 1+\varphi <  \frac{1}{m}(s_1 - s_0  ) \ln 2$ is used, thanks to   \eqref{log2}.

Following  the details in computing   \eqref{m322}  and \eqref{gradv} there exists   a positive constant $ I_u$ such that
\begin{equation}
\label{Itilde}
\begin{aligned}
\iint_{Q_{\frac R4}(\hat \theta)} \xi^2 |\nabla v|^2dxdt \leq 
 I_u\Big( \int_{-\theta}^{0}( |A_{k, \frac R4}|^{\frac {\tilde r} {\tilde q}}dt \Big)^{\frac{2(1+k)} {\tilde r}}.
\end{aligned}
\end{equation}

Thus inserting \eqref{Itilde}  in \eqref{1J1} ,    we have  \begin{equation}
\label{J12}
\begin{aligned}
&\int_{-\hat\theta}^{0}(\hat J_{12}+ \hat J_{21})dt \\
&\leq \mu_+^{\frac{ 2\mathfrak q-2} { m} } \frac{3}{m}(s_1 - s_0  )\ln2 \ \left( \frac{2^{s_1}}{\omega}\right)^2  I_u \ 
 \Big( \int_{-\hat \theta}^{0}( |A_{k, \frac R4}|^{\frac {\tilde r} {\tilde q}}dt \Big)^{\frac{2(1+k)} {\tilde r}}.
\end{aligned}
\end{equation}

Now we estimate $\hat L_2.$
 Since $ \phi \in L^{q,r}_{\mathbb{R}^N\times (t>0)},$  applying H\"older inequality, the following estimate holds
\begin{equation}
\label{L2}
\begin{aligned}
&\int_{-\hat\theta}^{0}\hat L_2 dt \leq 2 \big[  (s_1-s_0)  \ln 2 \big(\frac{2^{s_1} } {\omega} \big)\big]
 ||\phi||_{L^{q,r} _{Q_{\frac R4}  (\hat \theta)} } \!\!
 \Big( \int_{-\hat \theta}^{0}( |A_{k, \frac R4}|^{\frac {\tilde r} {\tilde q}}dt \Big)^{\frac{2(1+k)} {\tilde r}}.
 \end{aligned}
\end{equation}

Using the estimate \eqref{3eq.A^+} applied to $K_{\frac R4}$
and adding  \eqref{J12} with    \eqref{L2} , we obtain
 \begin{equation}
 \label{J2final}
\begin{aligned}
&\int_{-\hat\theta}^{0}(\hat J_{12}+\hat  J_{21}+L_2)dt \\
&\leq (s_1 - s_0)\ln 2  \frac{2^{s_1}}{\omega} 
\Big(  \mu_+^{\frac{2q-2}m} \frac 3m  \frac{2^{s_1}}{\omega} I_u+2
  ||\phi||_{L^{q,r}_{Q_{\frac R4} }} \Big)
|K_{\frac R4}|\,\Big(\frac{\omega}{A}\Big)^{\frac{2 \alpha(1+k) } {\tilde r }}
\end{aligned}  
\end{equation}

Inserting   \eqref{log3}, \eqref{I2final} and   \eqref{J2final} into \eqref{h}, we obtain
 \begin{equation}
 \label{ln}
\begin{aligned}
&\Big(\frac{1}{2m}\Big)^2 (s_1-s_0)^2  \ln^22 \  |\hat{P}|\\
&\leq 
2m C_1 (s_1-s_0) \ln2 \frac{2^6 } {R^2 }|K_{\frac R4}|\,\Big(\frac{\omega}{A}\Big)^{\frac{2 \alpha(1+k) } {\tilde r }}\\
&  
+(s_1 - s_0)\ln 2  \frac{2^{s_1}}{\omega} 
\Big(  \mu_+^{\frac{2q-2 }m} \frac 3m  \frac{2^{s_0}}{\omega}  I_u+2
  ||\phi||_{L^{q,r}_{Q_{\frac R2} }} \Big)
|K_{\frac R4}|\,\Big(\frac{\omega}{A}\Big)^{\frac{2 \alpha(1+k) } {\tilde r }}
\end{aligned}  
 \end{equation}
 
Dividing by $\frac{\hat C}{m^2} (s_1-s_0)^2 \  log^2 2 $,
choosing $A$  and $s_1$  sufficiently large, we conclude
 
 \begin{equation}
\begin{aligned}
&|\hat P|  \leq 
\Bigg\{\frac {2m^3 C_1  \frac{2^6 } {R^2 }}{\hat C(s_1-s_0) \ln2}\\
&  
+\frac {(m)^2 \frac{2^{s_1}}{\omega} }{\hat C(s_1 - s_0)\ln 2 }
\Big(  (\mu_+)^{\frac{2q-2 }m} \frac 3m  \frac{2^{s_1}}{\omega} I_u +2
  ||\phi||_{L^{q,r}_{Q_{\frac R2} }} \Big)\Bigg\}
|K_{\frac R4}|\,\Big(\frac{\omega}{A}\Big)^{\frac{2 \alpha(1+k) } {\tilde r }}\\
& =\nu_1 |K_{\frac R4}|.
\end{aligned}  
 \end{equation}

%appendix A6

\section{\normalsize {Estimate \eqref{IJL1L2}   in  Lemma \ref{log} in the $2^{nd}$ alternative }}
\label{cap:stimaIJL1L2}
Our aim is to estimate $ \int_{t^*}^{0}( -I +J +L_1+L_2) dt$.  
\begin{equation*}
 \begin{aligned}
 & -I =    
  -\int_{K_{ R}}m u^{m-1}\nabla u \cdot\left[ ( 2m(m-1)u^{m-2}(\nabla u) \varphi\varphi'\xi^2\,+\, 2mu^{m-1}((\nabla\varphi)\varphi' \xi^2\right.\\  
  &\left. + 2mu^{m-1}\varphi(\nabla\varphi')\xi^2 +4mu^{m-1}|\nabla u|\varphi \varphi'\xi\nabla\xi  \right] dx \leq -2(m-1)\int_{K_{ R}} u^{-1}(\nabla u^m|)^2 \varphi\varphi'\xi^2 dx \\
 &-2m\int_{K_{ R}}u^{m-1} (1+\varphi)\varphi'^{2}\xi^2\, |\nabla u^m|^2dx +4m \!\int_{K_{ R}}u^{m-1}|\nabla u^m|\varphi\varphi'\xi\,|\nabla \xi| dx= -I_1 -I_2 + I_3.
  \end{aligned}
  \end{equation*}
  
Moreover, by Young inequality we have
\begin{equation*}
  \begin{aligned}
& I_3 =4m\! \int_{K_{ R}}\! \!u^{m-1}\varphi(\frac{1}{\sqrt{2}}\varphi'\nabla u^m)(\sqrt{2}\nabla\xi) \! dx \leq 4m\int_{K_{ R}}u^{m-1}\varphi\left[ \frac{1}{4}\varphi'^2\xi^2\,|\nabla u^m|^2 +| \nabla \xi|^2\right] \!dx =\\
&m\int_{K_{ R}} \!\! u^{m-1} \varphi\varphi'^{2}\xi^2 |\nabla u^m|^2dx + 4m\int_{K_{ R}}u^{m-1}  \varphi\, |\nabla \xi|^2dx =I_{31} +I_{32}\\
\end{aligned}
  \end{equation*} 
\begin{equation*}
  \begin{aligned}
& J \!=\! \int_{K_{R}} \! u^{q-1}\nabla v\cdot\left[ 2m(m-1)u^{m-2}(\nabla u) \varphi\varphi'\xi^2+ 2mu^{m-1} (\nabla\varphi) \varphi' \xi^2\right.\\
&\left. + 2mu^{m-1}\varphi(\nabla\varphi')\xi^2 +4mu^{m-1}\varphi \varphi'\xi\nabla\xi  \right]dx \leq
2(m-1)\int_{K_{R}} \!\! u^{\mathfrak q-2}\varphi\varphi'\xi^2|\nabla u^m||\nabla v|dx \\
&+2m\int_{K_{R}}u^{\mathfrak q+m-2}(1+ \varphi)\varphi'^{2} \xi^2|\nabla u^m||\nabla v|dx 
+4m\int_{K_{R}} u^{\mathfrak q+m -2}\varphi\varphi'\xi |\nabla v||\nabla \xi|dx \\
&=J_1 + J_2 +J_3.
\end{aligned}
\end{equation*}
Using  Young's inequality and the fact that $\varphi'>1$
\begin{equation*}
\begin{aligned}
& J_1 
\leq (m-1)\int_{K_{R}}u^{-1}\varphi\varphi'\xi^2|\nabla u^m|^2dx +(m-1)\int_{K_{R}}\varphi\varphi'^2\xi^2|\nabla v|^2 u^{2\mathfrak q -3}dx =\\
& J_{11} + J_{12}.
 \end{aligned}  
 \end{equation*}
Using   Young's inequality  in $J_2$ we obtain
 \begin{equation*}
\begin{aligned}
&J_2\leq
2m\int_{K_{R}}  u^{m-1}(1+ \varphi)(\varphi^{2})' \xi^2 \! \left[\frac{1}{2} |\nabla u^m|^2 + \frac{1}{2}u^{2\mathfrak q-2}|\nabla v|^2\right] \\
&=m\int_{K_{R}} \!\! u^{m-1}(1+ \varphi)(\varphi^{2})' \xi^2|\nabla u^m|^2 + m \!\! \int_{K_{R}}  \!\! u^{m}(1+\varphi)\varphi'^{2} u^{2\mathfrak q-3}\xi^2|\nabla v|^2\\
&=J_{21} + J_{22},
\end{aligned}  
 \end{equation*}
 \begin{equation*}
\begin{aligned}
&J_3 =
4m\int_{K_{R}} u^{m -1}\varphi\left[ (\varphi'u^{\mathfrak q-1}\xi |\nabla v|)(|\nabla \xi|)\right] \leq\\
&2m\int_{K_{R}} u^{m }\varphi\varphi'^2 u^{2\mathfrak q-3} \xi^2 |\nabla v|^2 +2m\int_{K_{R}} u^{m-1 }\varphi|\nabla \xi|^2 
 = J_{31} + J_{32}
\end{aligned}  
 \end{equation*}
Note  that
$  -I +J +L_1+L_2 \leq 
(-I_{1} +J_{11} ) + (-I_{2} +J_{21} +I_{31}+L_1)$\\$+(I_{32}+ J_{32}) + (J_{12}+ J_{22}+J_{31})+L_2 .$\\
Now we see that
 \begin{equation*}
\begin{aligned}
&-I_{1} +J_{11}= -(m-1)\int_{K_{R}} u^{-1}|\nabla u^m|^2 \varphi\varphi'\xi^2 dx<0
\end{aligned}  
 \end{equation*}
and  

  \begin{equation*}
\begin{aligned}
&-I_{2} + J_{21} +I_{31}+ L_1 \! \leq - m   (1- 2C) \int_{K_{R}}u^{m-1} \varphi'^{2}\xi^2\, |\nabla u^m|^2dx, 
\end{aligned}  
 \end{equation*}
 and taking  \,\,$ C < 1/2 $, we  obtain a negative term that can be neglected.

 %%%
 \begin{equation}
\begin{aligned}
\label{I32J}
&I_{32}+ J_{32}=
4m\int_{K_{R}}u^{m-1}  \varphi\, |\nabla \xi|^2dx +2m\int_{K_{R}} u^{m-1 }\varphi|\nabla \xi|^2 \\
&=6m\int_{K_{R}}u^{m-1}  \varphi\, |\nabla \xi|^2dx. 
\end{aligned}  
 \end{equation}
%%%%%%%%%%%%%
 \begin{equation*}
\begin{aligned}
\label{J}
&J_{12}+ J_{22}+J_{31}\\
&=(m-1)\int_{K_{R}} \!\! \! \varphi\varphi'^2 u^{2\mathfrak q -3} \xi^2|\nabla v|^2 dx +\!  m \!\int_{K_{R}}\!\! u^{m}(1+\varphi)\varphi'^{2} u^{2\mathfrak q-3}\xi^2|\nabla v|^2dx \\
&+2m\int_{K_{R}} u^{m }\!\! \varphi\varphi'^2 u^{2\mathfrak q-3} \xi^2 |\nabla v|^2dx < 3m\int_{K_{R}}\varphi\varphi'^2 u^{2\mathfrak q -3} \xi^2|\nabla v|^2 dx \\
&+ 3m\int_{K_{R}}\!\! \!u^{m }(1+\varphi)\varphi'^2 u^{2\mathfrak q-3} \xi^2 |\nabla v|^2dx\leq 3m\int_{K_{R}} \!\! \!(1+u^{m })(1+\varphi)\varphi'^2 u^{2\mathfrak q-3} \xi^2 |\nabla v|^2dx.
\end{aligned}  
 \end{equation*}
 %%%%%%%%%%
From the  last inequalities, taking in account \eqref{eq.1} and \eqref{log2}
 we deduce
 \begin{equation*}
\begin{aligned}
&\int_{Q_{R}(t^*)} (1+u^{m })(1+\varphi)\varphi'^2 u^{2\mathfrak q-3} \xi^2 |\nabla v|^2dx dt \leq \\
&(1+\mu_+)(1+(s_1 - s_0 )\ln2)\left( \frac{2^{s_1}}{\omega}\right)^2 \!\! (\mu_+)^{\frac{2\mathfrak q-3}{m}} \int_{Q_{R}(t^*)} \!\!\xi^2 |\nabla v|^2dx dt.  \\
\end{aligned}  
\end{equation*}
Following  the details in computing  \eqref{gradv}, there exists  $I_u$ such that
\begin{equation*}
\label{ineqA }
\begin{aligned}
\int_{Q_{R}(t^*)} \xi^2 |\nabla v|^2dx dt \leq 
 I_u \left( \int _{t^*}^{0}\Big| A_{k,R}(t)\Big|^{\frac{\tilde{r}}{\tilde{q}}}dt\right)^{\frac{2}{\tilde{r}}(1 +\kappa)}
\end{aligned}
\end{equation*}
where $\Big| A_{k,R}(t)\Big|= \Big|\left\lbrace x \in K_{R}: u^m >\mu_+ \ -  \frac{\omega}{2^{s_1}}\right\rbrace \Big|.$ 

Thus  by using \eqref{eq.A^+} and  \eqref{eq.rtil} we have 
\begin{equation*}
 \begin{aligned}
 &\int_{t^*}^{0}(J_{12}+ J_{22}+J_{31})dt \\
 & \leq 3m I_u  (1+\mu_+ )(1+(s_1 - s_0 )\ln2) \frac{2^{2s_1}}{\omega^2}  \mu_+^{\frac{2\mathfrak q-3}{m}}|K_{R}|\,\left(\frac{\omega}{A}\right)^{\alpha(1-\frac{1}{r})}.
 \end{aligned}
\end{equation*}
%%%%%%%%%%%%%%%%%%%%%%%
Also we have  by \eqref{I32J}
 \begin{equation}
 \label{I32}
\begin{aligned}
&\int_{t^*}^{0}(I_{32}+ J_{32})dt
 \leq 6m(s_1 -s_0 )\ln2\,(\mu_+)^{m-1}\gamma \ |K_{R}| \  \theta_0^{-\alpha}\\
&\leq 6 \tilde \gamma m  \ \frac{(2^{s_0})^\alpha(s_1 -s_0)\ln2\,\mu_+^{m-1}}{\omega^\alpha }\  \Big| K_{R} \Big|.
\end{aligned}  
 \end{equation}
Now we estimate $L_2.$
 Since $ \phi \in L^{q,r}( \mathbb{R}^N\times (t>0)),$  applying H\"older inequality we have
\begin{equation}
\label{L2sec}
\begin{aligned}
&\int_{t^*}^{0}L_2dt\leq 2m [ \mu_+^{m-1} (s_1-s_0)  ln2 (\frac{2^{s_1} } {\omega})]
 ||\phi||_{{q,r,Q_{R}  (\hat \theta)} } 
 \Big( \int_{-\theta}^{0}( |A_{k,R}|^{\frac {\tilde r} {\tilde q}}dt \Big)^{\frac{2(1+k)} {\tilde r}}.
 \end{aligned}
\end{equation}

Adding \eqref{L2sec} to \eqref{I32}, 
 \begin{equation*}
\begin{aligned}
&\int_{t^*}^{0}(J_{12}+ J_{22}+J_{31}+L_2)dt \\
&\leq 3m\,I_u(1+\mu_+ )
(1+(s_1 - s_0)\ln2 )\Big( \frac{2^{s_1}}{\omega}\Big)^2  \mu_+^{\frac{2\mathfrak q-3}{m}}  |K_{R}|\,\Big(\frac{2^{s_0}}{A}\Big)^{\frac{2 \alpha(1+k) } {\tilde r } } 
 \\
&+ 2m \Big[ \mu_+^{m-1} (s_1-s_0)\ln2 \Big( \frac{2^{s_1} } {\omega}\Big)\Big] \
 ||\phi||_{{q,r, Q_{ R}  (\hat \theta)} } \Big( \frac{2^{s_0} }{ A}\Big)^{\frac{2 \alpha(1+k) } {\tilde r } } |K_{R}|. 
\end{aligned}  
 \end{equation*}
At the end we obtain the final estimate
  \begin{equation*}  
\begin{aligned} 
&\int_{t^*}^{0}( -I +J +L_1+L_2) dt\\
&\leq 3m\,I_u(1+\mu_+ )
(1+(s_1 - s_0)\ln2 )\Big( \frac{2^{s_1}}{\omega}\Big)^2 \mu_+^{\frac{2\mathfrak q-3}{m}} \Big(\frac{2^{s_0}}{A}\Big)^{\frac{2 \alpha(1+k) } {\tilde r } } |K_R| \\
&+ 2m \Big[ \mu_+^{m-1} (s_1-s_0)\ln 2 \Big( \frac{2^{s_1} } {\omega}\Big) \Big] || \phi||_{q,r,Q_{R}(\hat \theta)} \ \Big( \frac{2^{s_0} }{ A}\Big)^{\frac{2 \alpha(1+k) } {\tilde r } }|K_R|
\\
&+ 6m \gamma \  \frac{(2^{s_0})^\alpha}{ \omega^\alpha}  \frac{(s_1 -s_0 )\ln2 \  \mu_+^{m-1}}{ \sigma^2} \  |K_R|. 
\end{aligned}  
\end{equation*}

\section*{Acknowledgments} The authors gratefully thank the anonymous reviewer for his helpful remarks  and suggestions, which greatly improved the clarity of  this paper.
The authors are members of the Gruppo Nazionale per l'Analisi Matematica, la Probabilit\`a e le loro Applicazioni (GNAMPA) of the Istituto Na\-zio\-na\-le di Alta Matematica (INdAM). M.M. is partially supported by the research project {\it Evolutive and Stationary Partial Differential Equations with a focus on biomathematics} (Fondazione di Sardegna 2019) and by the grant PRIN n.PRIN-2017AYM8XW: {\it Non-linear Differential Problems via Variational, Topological and Set-valued Methods}.

\end{document}